\documentclass[12pt]{amsart}

\usepackage[T1]{fontenc}
\usepackage{amsmath,amsthm,amscd,amssymb}
\usepackage{graphicx}
\usepackage{psfrag,color}
\usepackage[ruled,lined,linesnumbered,norelsize]{algorithm2e} 

\newtheorem{theorem}{Theorem}[section]
\newtheorem{lemma}[theorem]{Lemma}

\newtheorem{definition}[theorem]{Definition}
\newtheorem{assumption}[theorem]{Assumption}

\def\RR{\mathbb{R}}
\def\Z{\mathbb{Z}}
\def\eps{\varepsilon}
\def\C{\mathcal{C}}

\def\K{\mathcal{K}}
\def\T{\mathcal{T}}
\def\L{\mathcal{L}}
\def\R{\mathcal{R}}

\title{Rigorous enclosures of a slow manifold}
\author{John Guckenheimer, Tomas Johnson, \& Philipp Meerkamp}
\address{Department of Mathematics, Cornell University, Ithaca, NY 14853, USA}
\email{jmg16@cornell.edu, tomas.johnson@cornell.edu, pmeerkamp@math.cornell.edu}
\keywords{Slow-fast system, invariant manifold, interval analysis}
\subjclass[2000]{34C45,34E13,34E15,37G25,37M99}
\date{2012-05-05}
\begin{document}

\begin{abstract}
Slow-fast dynamical systems have two time scales and an explicit parameter representing the ratio of these time scales. Locally invariant slow manifolds along which motion occurs on the slow time scale are a prominent feature of slow-fast systems. This paper introduces a rigorous numerical method to compute enclosures of the slow manifold of a slow-fast system with one fast and two slow variables. A triangulated first order approximation to the two dimensional invariant manifold is computed ``algebraically''. Two translations of the computed manifold in the fast direction that are transverse to the vector field are computed as the boundaries of an initial enclosure. The enclosures are refined to bring them closer to each other by moving vertices of the enclosure boundaries one at a time. As an application we use it to prove the existence of tangencies of invariant manifolds in the problem of singular Hopf bifurcation and to give bounds on the location of one such tangency.
\end{abstract}

\maketitle

\tableofcontents

\section{Introduction}\label{S_Intro}

Invariant manifolds and their intersections are important features that organize qualitative properties of dynamical systems. Three types of manifolds have been prominent in the subject: (1) compact invariant tori ~\cite{Dl01}, (2) stable and unstable manifolds of equilibria and periodic orbits~\cite{Kea05,EKO07}, and (3) slow manifolds of multiple time scale systems~\cite{J95}. Interval arithmetic and verified computing have been used extensively to give rigorous estimates and existence proofs for invariant tori and occasionally to locate stable and unstable manifolds, but this paper is the first to employ these methods to locate slow manifolds. Each of these three cases pose numerical challenges to locate the manifolds. 

Many methods that locate invariant tori assume that the flow on the tori is smoothly conjugate to a constant flow with dense orbits. Existence of this conjugacy confronts well known small divisor problems and the winding vector of the flow must satisfy diophantine conditions in order for this problem to be solvable. Typically, the numerical methods produce a Fourier expansion of the conjugacy which is determined up to a translation. The manifolds are located by projection onto a discrete set of Fourier modes and solving a fixed point equation for the coefficients of the conjugacy.

The computation of stable and unstable manifolds of equilibria and periodic orbits is a ``one-sided'' boundary value problem. The manifolds consist of trajectories that are asymptotic to the equilibrium or periodic orbit. In the case of an equilibrium point of an analytic vector field, the local stable and unstable manifolds are analytic graphs that have convergent asymptotic expansions whose coefficients can be determined iteratively. The most challenging aspect of computations of two dimensional manifolds arises from the way that trajectories do or do not spread out in the manifold as one departs from the equilibrium or periodic orbit. As illustrated by the Lorenz manifold~\cite{Kea05}, the manifolds can twist and fold in ways that present additional geometrical complications for numerical methods. The development of rigorous bounds for these invariant manifolds follows similar principles to the verified computation of individual trajectories.

Multiple time scale vector fields, also known as singularly perturbed differential equations, occur in many settings: systems of chemical reactions, lasers, fluid dynamics and models of the electrical activity of neurons are a few examples. Borrowing terminology from fluid dynamics, the solutions of these systems can have (boundary) layers in which the fast time scale determines the rate at which the solution varies as well as long periods of time during which the solution evolves on the slow time scale. The slow motion typically occurs along \textit{slow manifolds} that are locally invariant. The slow manifolds play a prominent role in qualitative analysis of the dynamics and bifurcations of multiple time scale systems. Indeed, model reduction procedures are frequently employed that replace a model by a lower dimensional model that approximates the motion along a slow manifold and ignores the fast dynamics of the original model. The ideal for this type of model reduction is an algorithm that computes the slow manifold exactly. That ideal seems very difficult to achieve and is not addressed in this paper. Instead, we seek rigorous bounds for the location of the slow manifold that are tight enough to give information that can be used in the analysis of bifurcations of the system. 

To explain the methods we introduce in the simplest terms, we focus upon \textit{slow-fast} systems that contain an explicit parameter $\eps$ that represents the ratio of time scales. Moreover, we restrict attention to systems that have two slow variables and one fast variable and use a single example as a test case. In principle, the methods generalize to the case of codimension one slow manifolds, and the definitions and existence proofs in Sections \ref{S_OverMethod} and \ref{S_Existence} have obvious higher dimensional analogues. In practice, however, due to the scarcity of tools for computational geometry in higher dimensions, implementing a higher dimensional version would be a significant extension of the work described in this paper. 
We comment on generalizations from the setting of systems with two slow and one fast variable in the discussion at the end of the paper, but leave consideration of further details to future work.

Slow manifolds of multiple time scale systems present unique theoretical and numerical challenges compared to the computation of invariant tori and (un)stable manifolds. The first of these challenges is that theory is developed primarily in terms of ``small enough'' values of the parameter $\eps$ measuring the time scale ratio of a slow-fast system. Numerically, one always works with specific values of $\eps$. The convergence of trajectories as $\eps \to 0$ is singular, making it difficult to develop methods framed in terms of asymptotic expansions in $\eps$. Divergent series are the rule rather than the exception in this context. The rich history of numerical integration methods for stiff systems and the large literature on reduction methods for kinetic equations of chemical systems reflect the difficulty of computing attracting slow manifolds, the simplest case for this problem. Computing slow manifolds of saddle-type presents the additional challenge that most nearby trajectories diverge from the slow manifold on the fast time scale in both forward and backward time. The second theoretical difficulty in finding slow manifolds is that they are only locally invariant in most problems of interest. The local invariance is accompanied by a lack of uniqueness: possible manifolds intersect fast subspaces in open sets whose diameter is exponentially small in $\eps$; i.e., bounded by $\exp(-c/\eps)$ for a suitable $c > 0$. Methods based upon root finding of a discretized set of equations must choose a specific solution of the discretized equations.

We compute enclosures of slow manifolds by exploiting transversality properties that improve as $\eps \to 0$ while being suitable for fixed values of $\eps$. The methods do not identify a unique object and are well suited to locating locally invariant slow manifolds. If $H$ is a hypersurface and $F$ is a vector field, then transversality of $F$ to $H$ is a \emph{local} property: verification does not rely upon computation of trajectories of $F$. For a slow-fast vector field with one fast variable, translation of a normally hyperbolic critical manifold along the fast direction produces a transverse hypersurface when the translation distance is large enough. Translation distances proportional to $\eps$ suffice. In this paper, we use piecewise linear surfaces $H$ as enclosing manifolds. For the example we consider, transversality at vertices of a face of $H$ implies transversality of the entire face. This reduces the computational complexity of checking transversality sufficiently that iterative refinement of the enclosures was feasible. 

Since slow manifolds are objects that are defined asymptotically in terms of the parameter $\eps$, they are not directly computable using finite information. One part of this paper is devoted to the development of a mathematical framework within which slow manifolds are defined for fixed values of $\eps > 0$. We define \emph{computable slow manifolds} and relate this concept to the slow manifolds studied in geometric singular perturbation theory. All computations and statements in this paper are for computable slow manifolds. This is similar in spirit to the finite resolution dynamics approach of Luzzatto and Pilarczyk \cite{LP11}. 

Our work is motivated by the study of tangencies of invariant manifolds. Significant global changes in the dynamics of a system have been observed to occur at bifurcations involving tangencies. Proving the existence of tangencies is intrinsically complicated because the manifolds themselves must be tracked over a range of parameters. Computer-aided proofs of tangencies of invariant manifolds have previously been studied by Arai and Mischaikow in \cite{AM06}, and Wilczak and Zgliczy\'nski in \cite{WZ09}. In Section \ref{S_Tang}, we prove that a tangency bifurcation involving a computable slow manifold occurs in the singular Hopf normal form introduced in \cite{G08}.   

\subsection{Slow-fast systems}\label{SS_SlowFast}
Slow-fast differential equations have the form:
\begin{eqnarray}\label{eq_slowFast}
\eps \dot x&=&f(x,y,\eps) \\
\nonumber \dot y & = & g(x,y,\eps),
\end{eqnarray}
where $x\in\RR^n$, $y\in\RR^m$, $f:\RR^{n+m+1} \rightarrow \RR^n$, and $g:\RR^{n+m+1} \rightarrow \RR^m$. We assume that the vector field $(f,g)$ is smooth ($C^\infty$), although most of this paper can easily be adapted to the finitely differentiable setting. Here $x$ and $y$ are the fast and slow variables, respectively. Throughout the  paper we consider the case $m=2$ and $n=1$ of two slow variables and one fast variable.

We define the \textit{critical manifold}, as the set
\begin{equation}\label{eq_critMfd}
S_0:=\{(x,y)\in \RR^{n+m} : f(x,y,0)=0\}.
\end{equation}
The critical manifold is normally hyperbolic at points where $D_xf$ is hyperbolic; i.e., has no eigenvalue whose real part is zero. Points where $S_0$ is singular are referred to as folds. On the normally hyperbolic pieces of the critical manifold, $x$ is given as a function of $y$, $x=h_0(y)$. The corresponding differential equation 
\begin{equation}\label{eq_slowSystem}
 \dot y = g(h_0(y),y,0)
\end{equation}
is called the slow flow. If one instead rescales time with $\eps$ and puts $\eps=0$ in \eqref{eq_slowFast}, one gets the \textit{layer equation}:
\begin{eqnarray}\label{eq_fastSystem}
x'&=&f(x,y,0) \\
\nonumber y' & = & 0,
\end{eqnarray}
Note that the manifold $S_0$ is exactly the set of critical points for the layer equation. Singular perturbation theory studies how the solutions to (\ref{eq_slowFast}) for $\eps$ small, but positive, can be understood by studying solutions to (\ref{eq_slowSystem}) and (\ref{eq_fastSystem}).

When $S_0$ is normally hyperbolic and $\eps>0$ is sufficiently small, geometric singular perturbation theory \cite{J95} ensures that the critical manifold perturbs to a \textit{slow manifold}. Slow manifolds are \textit{locally invariant} and $O(\eps)$ close to the critical manifold. However, slow manifolds are not unique, although different choices are within $O(e^{-c/\eps})$ distance from each other. We denote slow manifolds by $S_\eps$.The purpose of this work is to compute approximations of $S_\eps$ that are guaranteed to be of a certain accuracy. This is achieved by computing two approximations that enclose the slow manifold. The two approximations of the slow manifold are triangulated surfaces transverse to the vector field. To prove the transversality, we use interval analysis, to be explained in Subsection \ref{SS_ValNum}. Interval analysis is a general technique that enables mathematically rigorous proofs of inequalities on a digital computer. 

To simplify notation we denote the two slow variables by $y$ and $z$, i.e., from now on $y\in\RR$, and the vector field in the slow variables is denoted by $g=(g_y,g_z)$. We also assume that $f,g_y,$ and $g_z$ are independent of $\eps$. To summarize, the systems we study are of the following form:
\begin{eqnarray}\label{eq_slowFast_12}
\eps \dot x&=&f(x,y,z) \\
\nonumber \dot y & = & g_y(x,y,z) \\
\nonumber \dot z & = & g_z(x,y,z),
\end{eqnarray}
where $x,y,z\in\RR$, and $f,g_y,g_z:\RR^{3} \rightarrow \RR$. We will sometimes use the notation $F=(f,g_y,g_z)$.

\subsection{Validated numerics}\label{SS_ValNum}
\textit{Interval analysis} was introduced by Moore in \cite{Mo66} as a method to use a digital computer to produce mathematically rigorous results with only approximate arithmetic. Tucker \cite{T11} is a modern introduction to the subject, and more advanced topics are discussed by Neumaier \cite{Ne90}. The main idea is to replace floating point arithmetic with a set arithmetic; the basic objects are intervals of points rather than individual points. Together with directed rounding this method yields an enclosure arithmetic that allows for the rigorous verification of inequalities. To use interval analysis to produce a mathematical proof, often called \textit{(auto-)validated numerical methods}, one has to prove that the statement at hand can be reduced to a finite number of inequalities, and then verify that these inequalities are satisfied. Interval arithmetic is used for the verification. The objects used to describe sets in validated numerics are typically convex sets in some coordinate system, e.g., intervals, parallelograms, or ellipsoids. In this paper we will employ triangular meshes of surfaces, an approach that previously, in this setting, only has been used in \cite{JT11c}.

\subsection{Computation of invariant manifolds}\label{SS_InvMfd}
The study of invariant manifolds \cite{HPS77} is central to the theory of dynamical systems. The behavior of a system can often be understood by understanding its invariant structures. Numerical computations of invariant manifolds \cite{CFL05,C09,CZ11,EKO07,GK09,JT11a,Kea05,O95,S90,Z09} are important in many applications. There are no universally applicable methods to compute invariant manifolds; to be efficient, they have to be tailored for the specific class of problems one is studying. Computing invariant manifolds of slow-fast systems is particularly challenging. Two existing methods are \cite{EKO07,GK09}, and no rigorous methods exist. The main idea of our method is to refine a first order approximation of the manifold by local modifications that maintain transversality of the enclosing manifolds. Interval arithmetic is used to make the local computation of transversality rigorous. This is similar in spirit to the methods developed in \cite{G95} to study the phase portraits of planar polynomial vector fields. Even in the planar case the verified computation of phase portraits is a challenging task, and the few methods that exist include \cite{G95,JT11b}. 


\section{Overview of the method}\label{S_OverMethod}
This section describes our method to compute enclosures of the slow manifold of a slow-fast system of the form (\ref{eq_slowFast_12}). We start by giving an overview of the main ideas of the method. There are five  main steps in the algorithm: 
\begin{enumerate}
 \item 
triangulation of the critical manifold,
\item
computing the $O(\eps)$ correction term for the slow manifold,
\item
constructing left and right perturbations of the slow manifold,
\item
proving that the left and right perturbations enclose the manifold, and
\item
tightening the enclosure by contracting the left and right perturbations towards each other.
\end{enumerate} 

The first step is to compute a triangulation of the critical manifold, which is adapted to its geometry. The manifold is defined implicitly by the condition $f(x,y,z)=0$. In the example we consider in  Section \ref{S_Method}, we solve this equation to obtain explicit expressions for the functions of the form $x = h_0(y,z)$ whose graphs lie in the critical manifold. Alternatively, one computes approximations to $h_0$ using, e.g., automatic differentiation and continuation procedures. There are many software packages to compute triangulations of surfaces; we use CGAL \cite{CGAL} via its matlab interface. When a part of the critical manifold is represented as the graph of a function $h_0$, its domain in the plane of the slow variables can be triangulated, and then this triangulation can be lifted to the graph, as illustrated in Figure \ref{F_triangulation}. So that the triangles in the lifted triangulation have similar diameters, we choose triangles in the plane of the slow variables to have diameters that depend upon the gradient of $h_0$. We stress that the rest of the algorithm is independent from how the triangulation of the critical manifold is constructed. Rather than using axis parallel patches, one could, e.g., use approximate trajectory segments of the reduced system to determine the piece of the domain of the slow variables, where the slow manifold is computed. 

\begin{figure}[h]
\begin{center}
(a)\includegraphics[width=0.45 \textwidth]{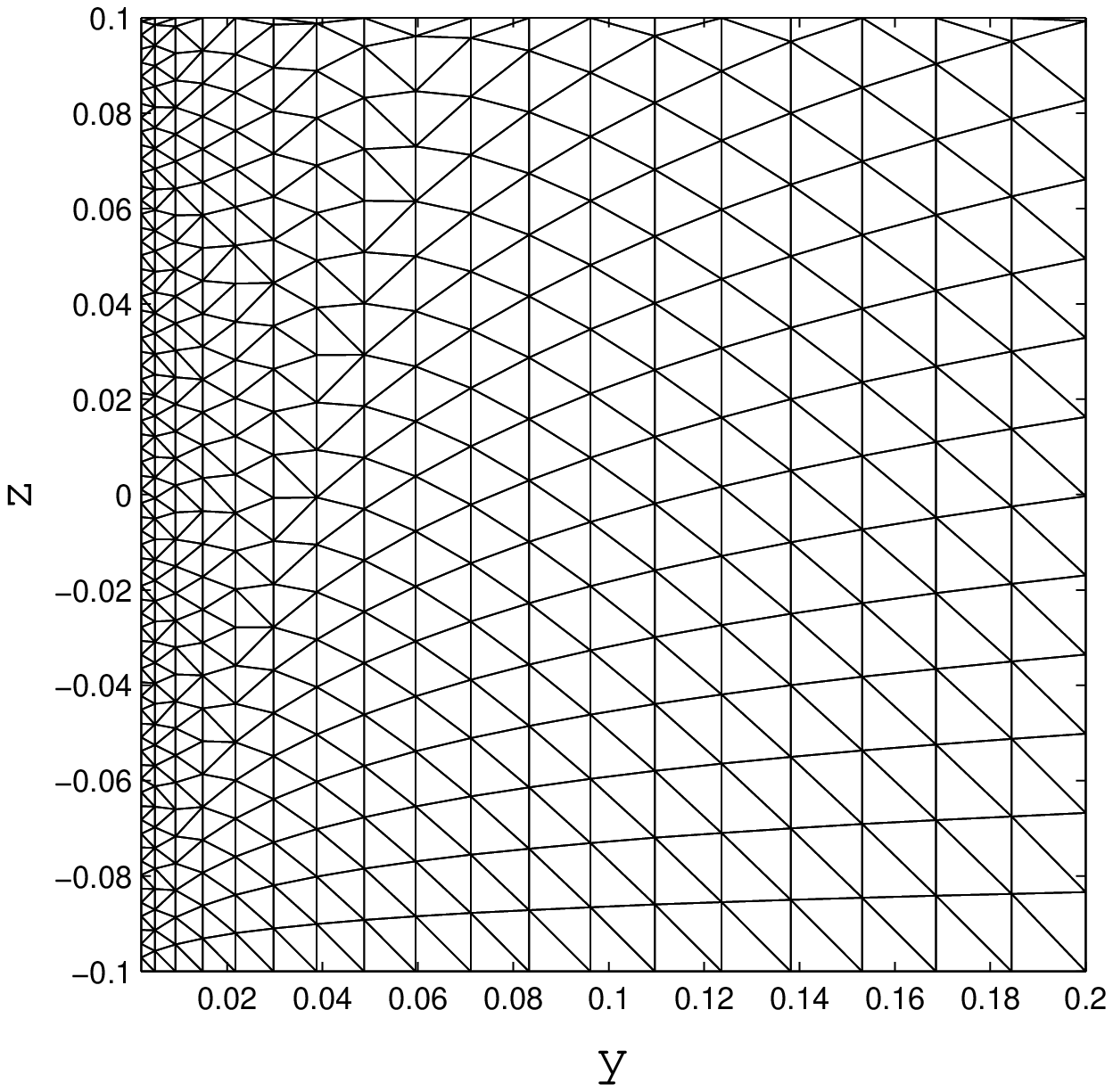}
(b)\includegraphics[width=0.45 \textwidth]{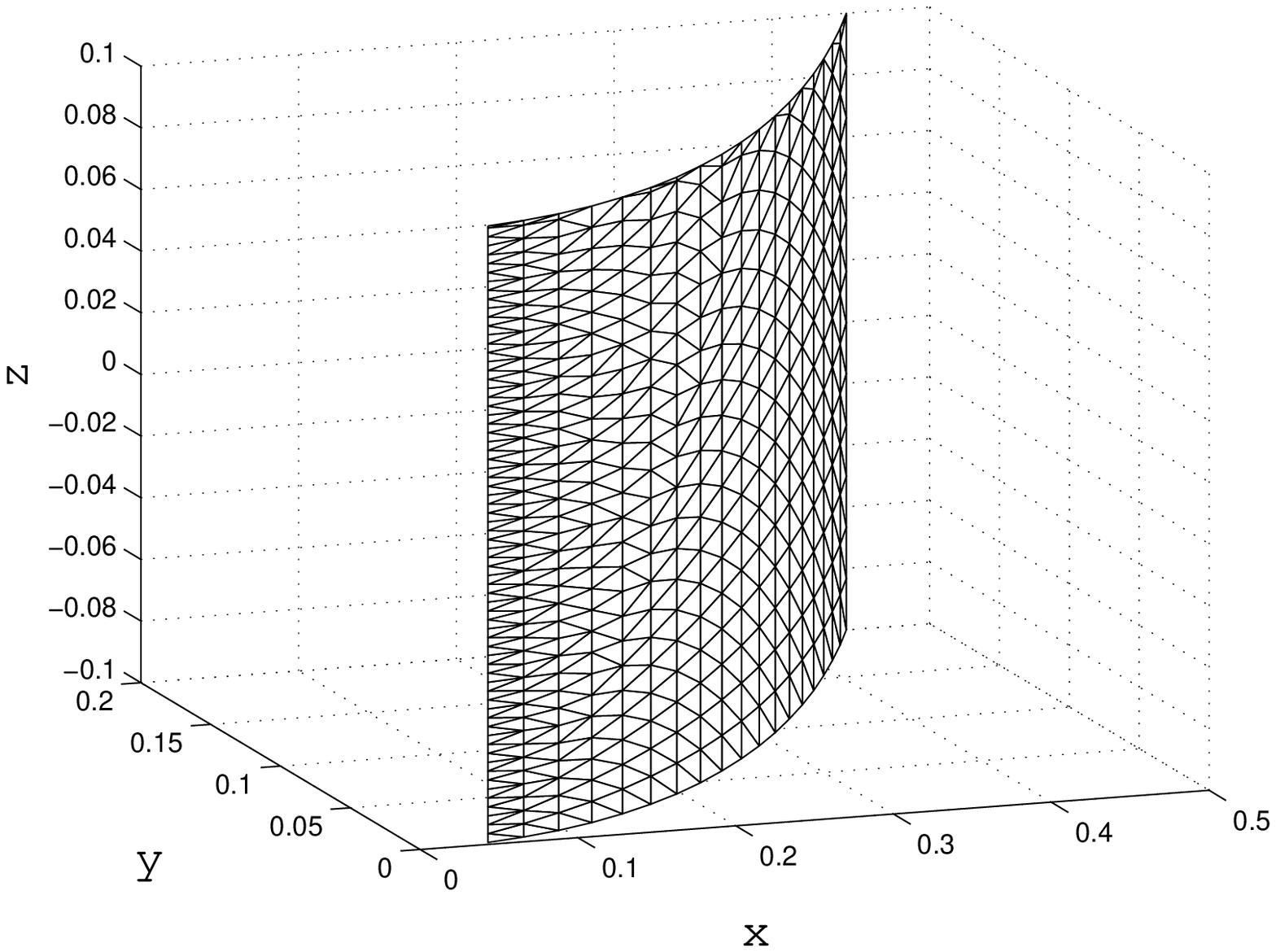}
\caption{The mesh generated for the example in Section \ref{S_Method}. There is a fold at $\{y=0\}$. (a) The Delaunay triangulation of the $(y,z)$ plane that is generated by the geometry adapted mesh points. (b) The lift of the triangulation to the critical manifold.}\label{F_triangulation}
\end{center}
\end{figure}

We compute an approximation to the slow manifold using a procedure similar to that employed in stiff integrators that use Rosenbrock methods \cite{HW}. The tangent space to the critical manifold is orthogonal to the vector $df$. According to the Fenichel theory, the slow manifold is $O(\eps)$ close to the critical manifold in the $C^1$ topology, so its tangent space is approximately normal to $df$. At a point $(x,y,z)$ in the (lifted) triangulation of $S_0$, we look for a nearby point $(x+\delta,y,z)$ at which the vector field is orthogonal to $df(x,y,z)$. Since $f(x,y,z) = 0$ and the normal hyperbolicity implies that $\partial_x f \ne 0$, 
$$\delta = -\eps \frac{(\partial_y f g_y + \partial_z f g_z)}{(\partial_x f)^2}$$ 
is an approximate solution to this equation. 
Setting $\delta$ to this value, we take $(x+\delta,y,z)$ as a point of the triangulation of the approximate slow manifold.
The critical manifold and the approximation to the slow manifold are illustrated in Figure \ref{F_critSlowMfd} (a) and (b), respectively. We next perturb this triangulation of the approximate slow manifold in both directions parallel to the $x$-axis, as in Figure \ref{F_critSlowMfd} (c), by a factor $2^{j-6}\delta$, where $j$ is a natural number that will be specified later. In case that $\delta$ is very small, we replace it by a $O(\eps^2)$ term. This procedure yields two surfaces that are candidates for the enclosing surfaces that we seek. 

\begin{figure}[h]
\begin{center}
(a)\includegraphics[width=0.28 \textwidth]{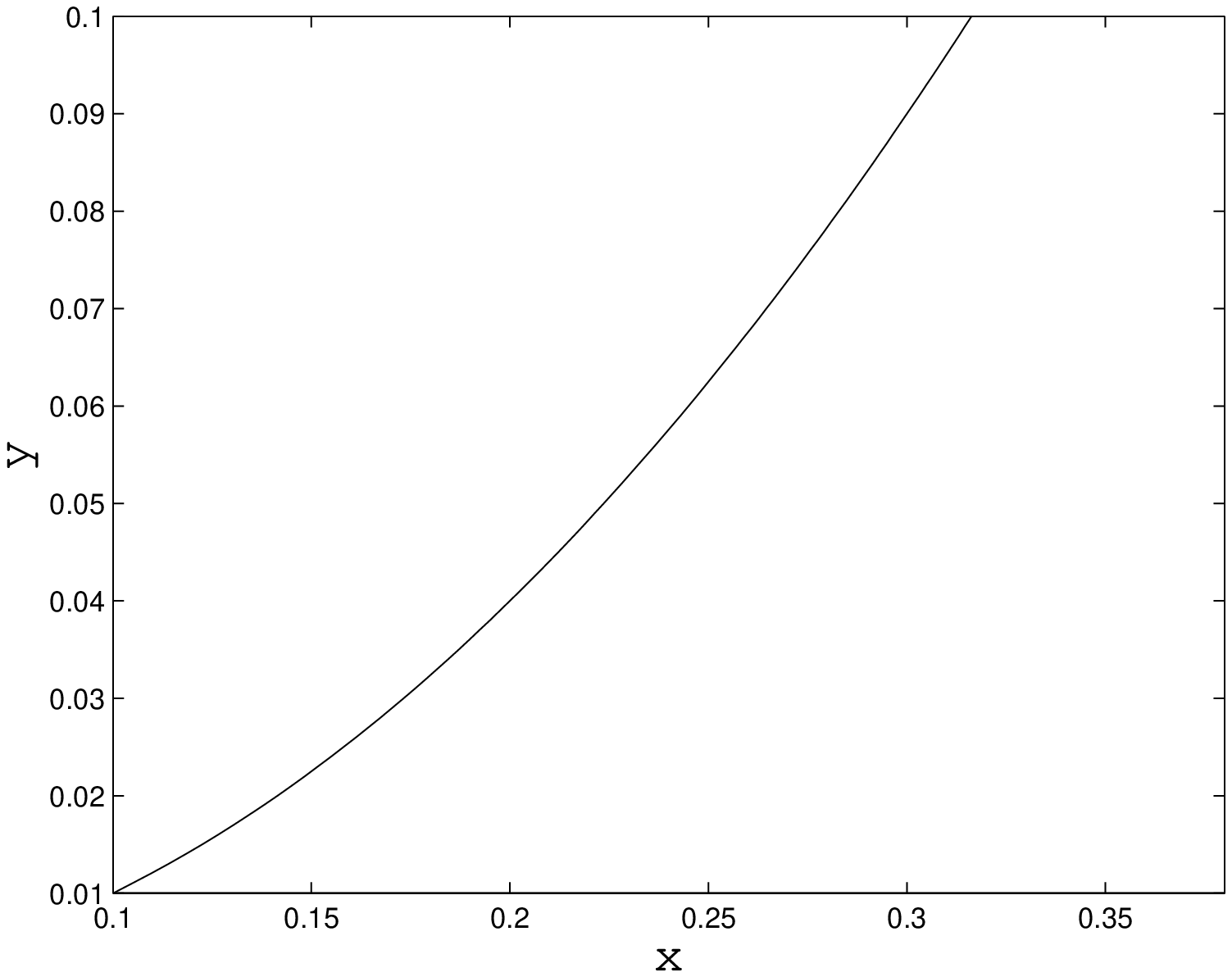}
(b)\includegraphics[width=0.28 \textwidth]{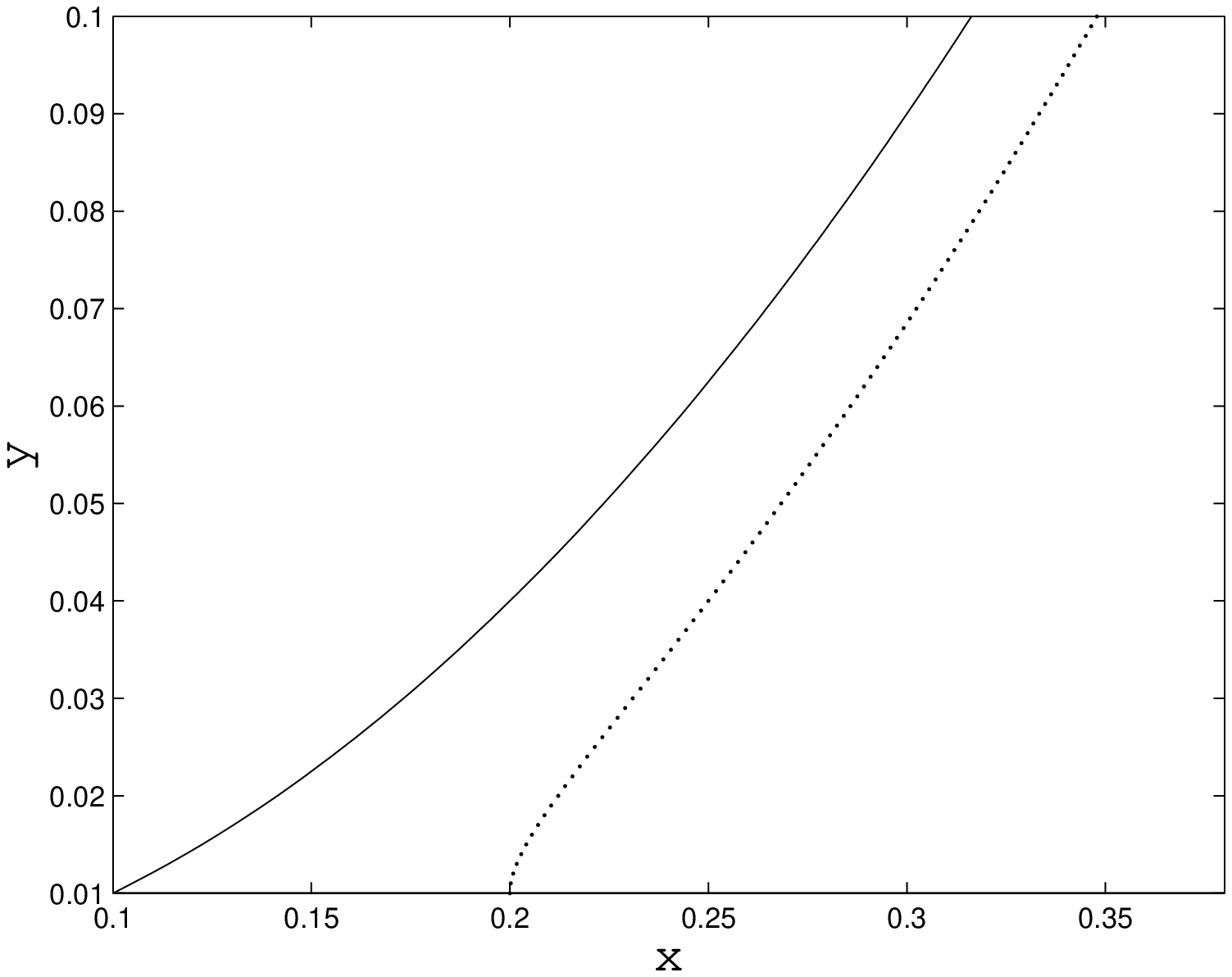}
(c)\includegraphics[width=0.28 \textwidth]{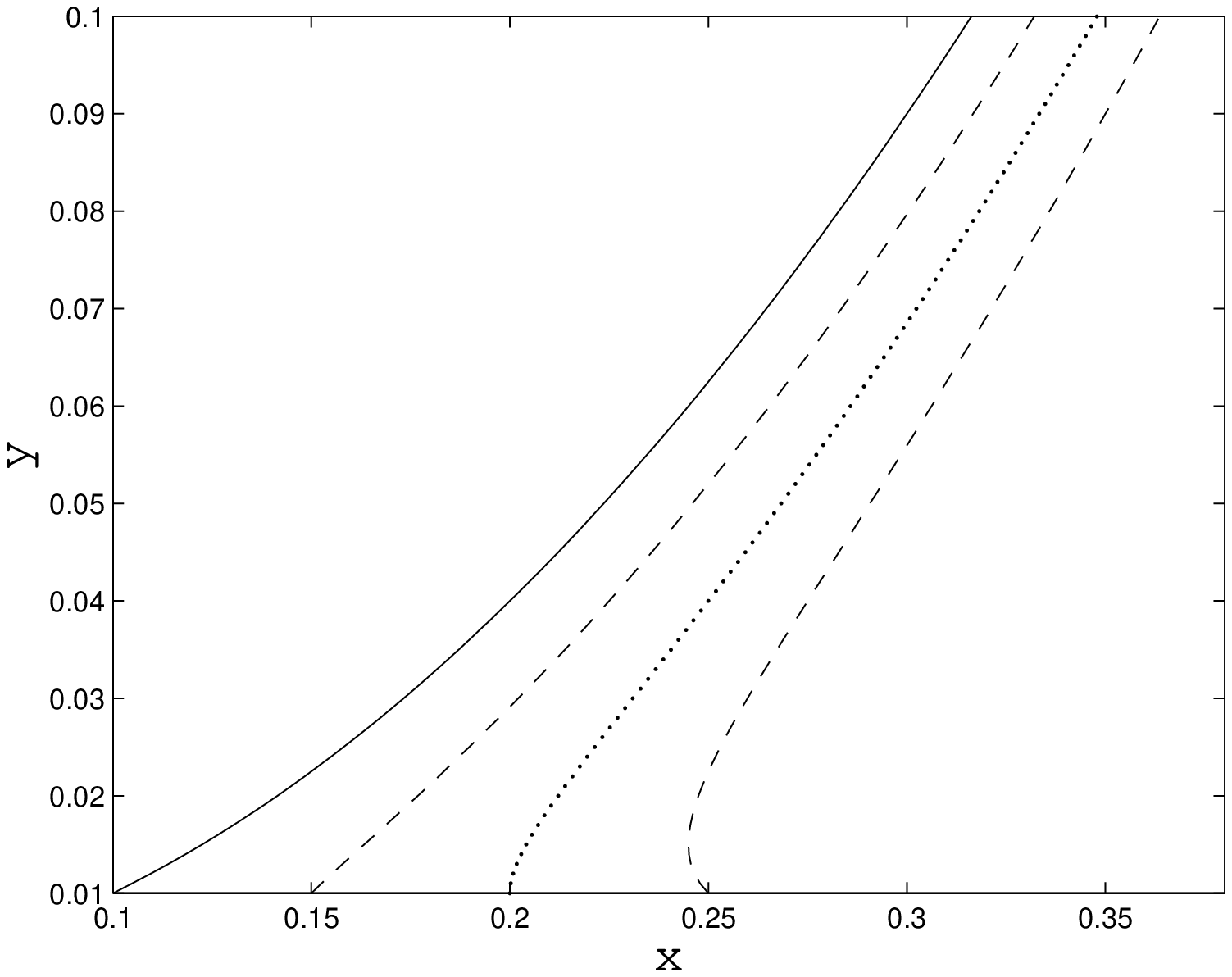}
\caption{Construction of the enclosing triangulations. This figure shows the projection on $(x,y)$ coordinates of: (a) the critical manifold (solid) (b) the critical manifold (solid), and the slow manifold (dotted) (c) the critical manifold (solid), the slow manifold (dotted), and the two enclosing surfaces (dashed).}\label{F_critSlowMfd}
\end{center}
\end{figure}

To verify that the surfaces enclose the slow manifold, we check whether the flow of the full system (\ref{eq_slowFast_12}) is transversal to the candidate surfaces. As the candidate surfaces are piecewise linear, we have to define what we mean by transversality at the edges and vertices of the triangulation. 

\begin{definition}\label{D_Cone}
Let $\T\subset \RR^3$ be a triangulated, piecewise linear two dimensional manifold $\T=\bigcup T_i$. Since $\T$ is a manifold, it locally separates $R^3$ into two sides. We say that a vector $v$ is transverse to $\T$ if $v$ and $-v$ point to opposite sides of $\T$. A smooth vector field is transverse to $\T$ if it is transverse to $\T$ at every point of $\T$. 
\end{definition}

Figure \ref{F_cone}(a) illustrates this definition. Trajectories of the flow generated by $F$ will all cross $\T$ from one side to another if $F$ is transverse to $\T$. If $\T_1$ and  $\T_2$ are triangulated surfaces transverse to the flow with opposite crossing directions, then they form enclosing surfaces for the slow manifold we seek. 

\begin{figure}[h]
\begin{center}
(a)\includegraphics[width=0.35 \textwidth]{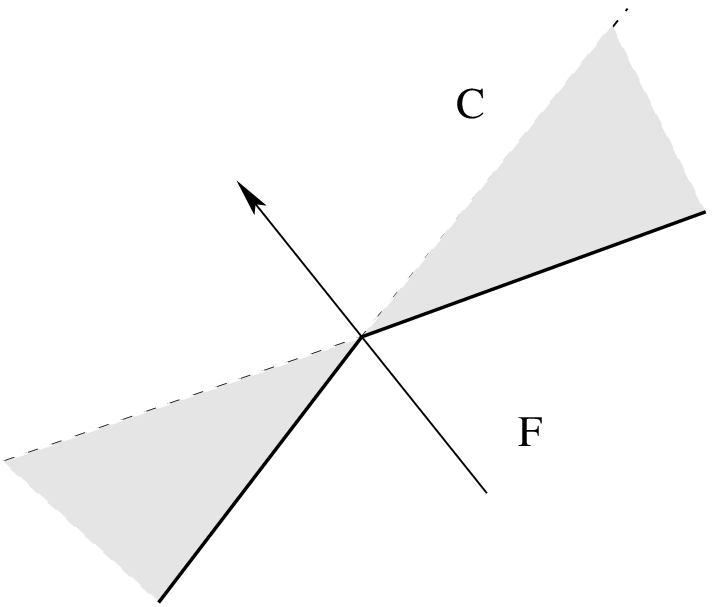}
\hspace{1.0cm}
(b)\includegraphics[width=0.2 \textwidth]{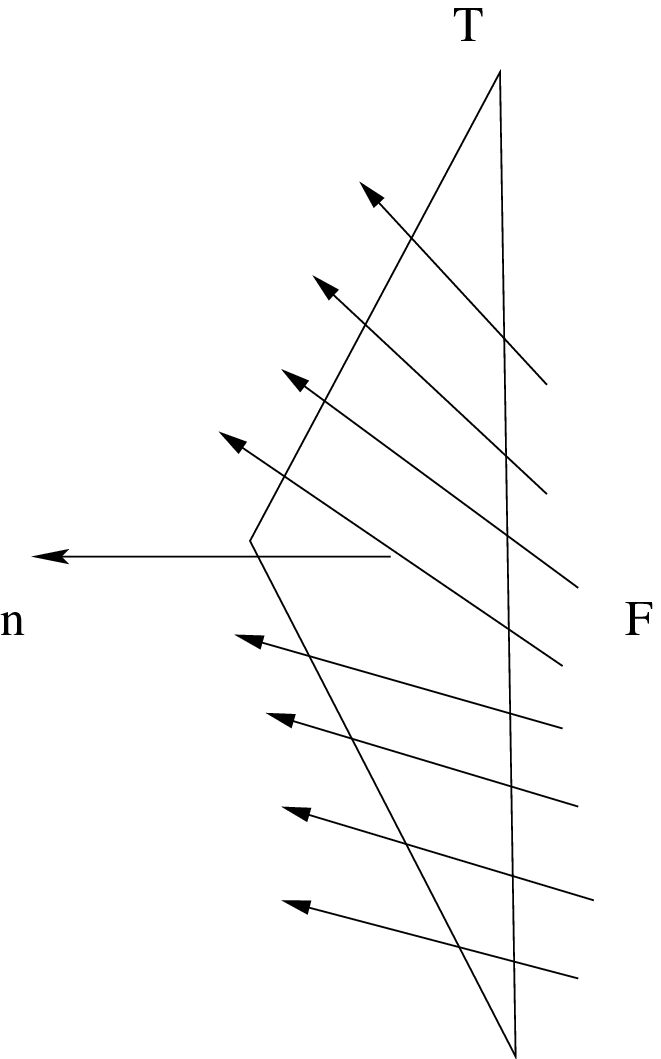}
\caption{(a) Transversality check on an edge of the triangulation. $F$ is transversal if it does not belong to the cone $\C$. (b) Transversality check on one face of the triangulation. To verify that the flow intersects the surface transversally, it suffices to prove that the vector field is never orthogonal to the normal of the surface, which a constant vector. So, we compute the range of $F\cdot n$, and prove that it does not contain $0$.}\label{F_cone}\label{F_triangle}

\end{center}
\end{figure}

%

Transversality is a condition that is local to each face of the triangulation, so we can check it on each face of the triangulation separately. To check the transversality condition on one face, we estimate the range of the inner product of the vector field with the normal of the face, as illustrated in Figure \ref{F_triangle}(b). Details about the existence of locally invariant, normally hyperbolic manifolds inside the enclosure are addressed in Section \ref{S_Existence} below.

%

The final part of the algorithm is to iteratively update the location of the vertices by moving them towards each other in small steps along the fast direction. We check that the transversality properties still hold, see Figure \ref{F_updateMfd}. This tightening step is stopped when no more vertices can be moved. Note that the vertices of all $4$ triangulations: the critical manifold, the approximate slow manifold, and the two perturbed manifolds, all have the same $(y,z)$ components.

\begin{figure}[h]
\begin{center}
(a)\includegraphics[width=0.28 \textwidth]{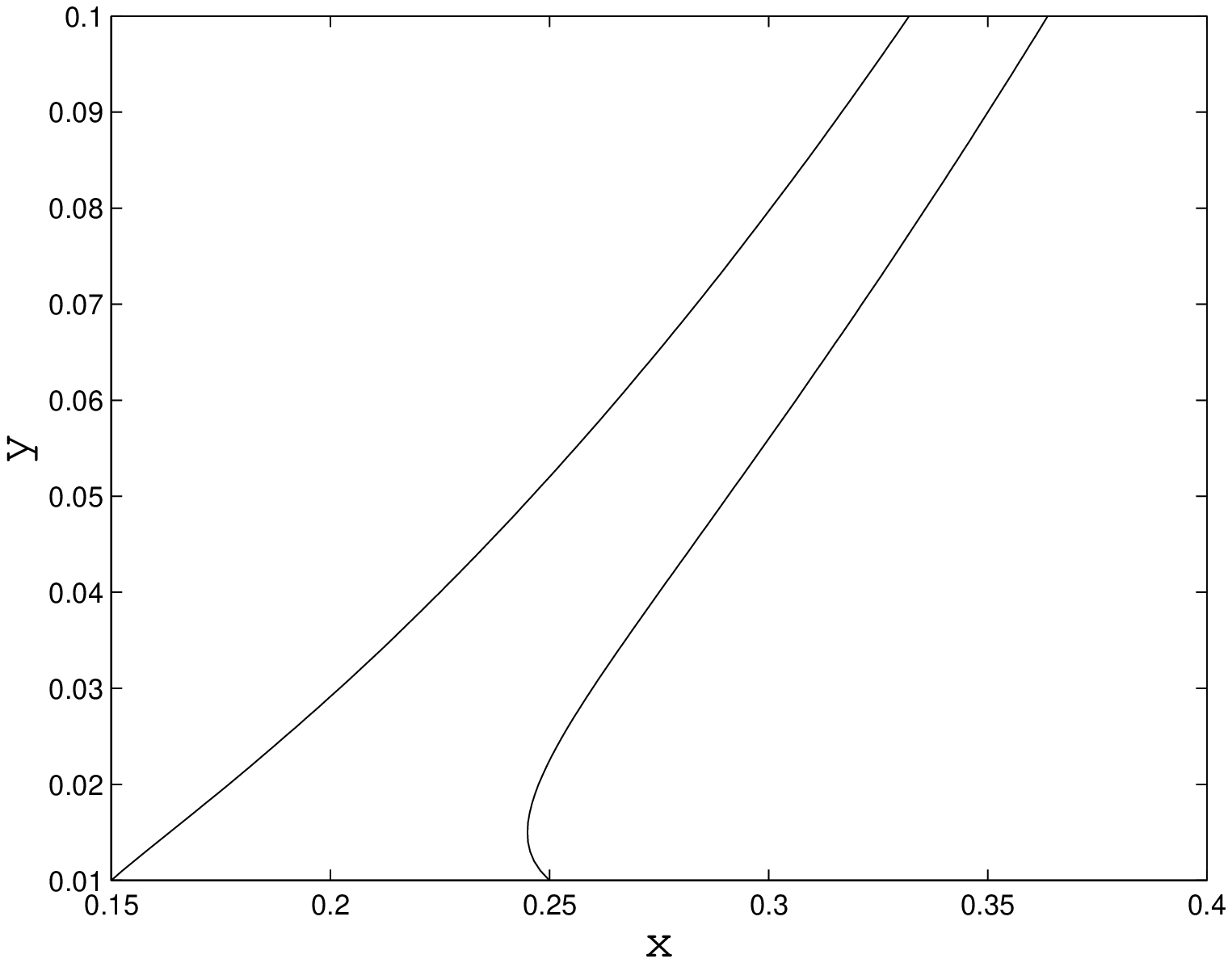}
(b)\includegraphics[width=0.28 \textwidth]{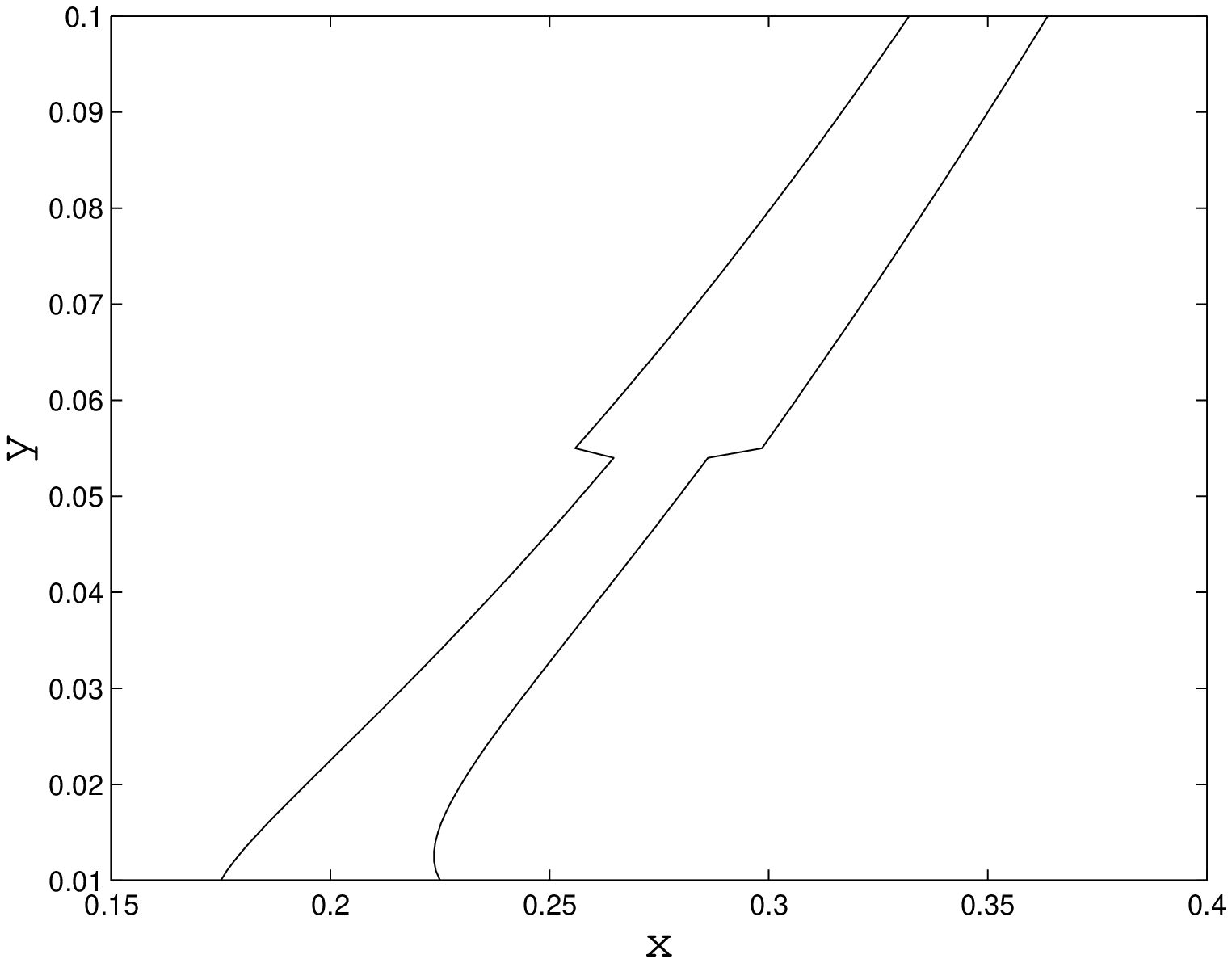}
(c)\includegraphics[width=0.28 \textwidth]{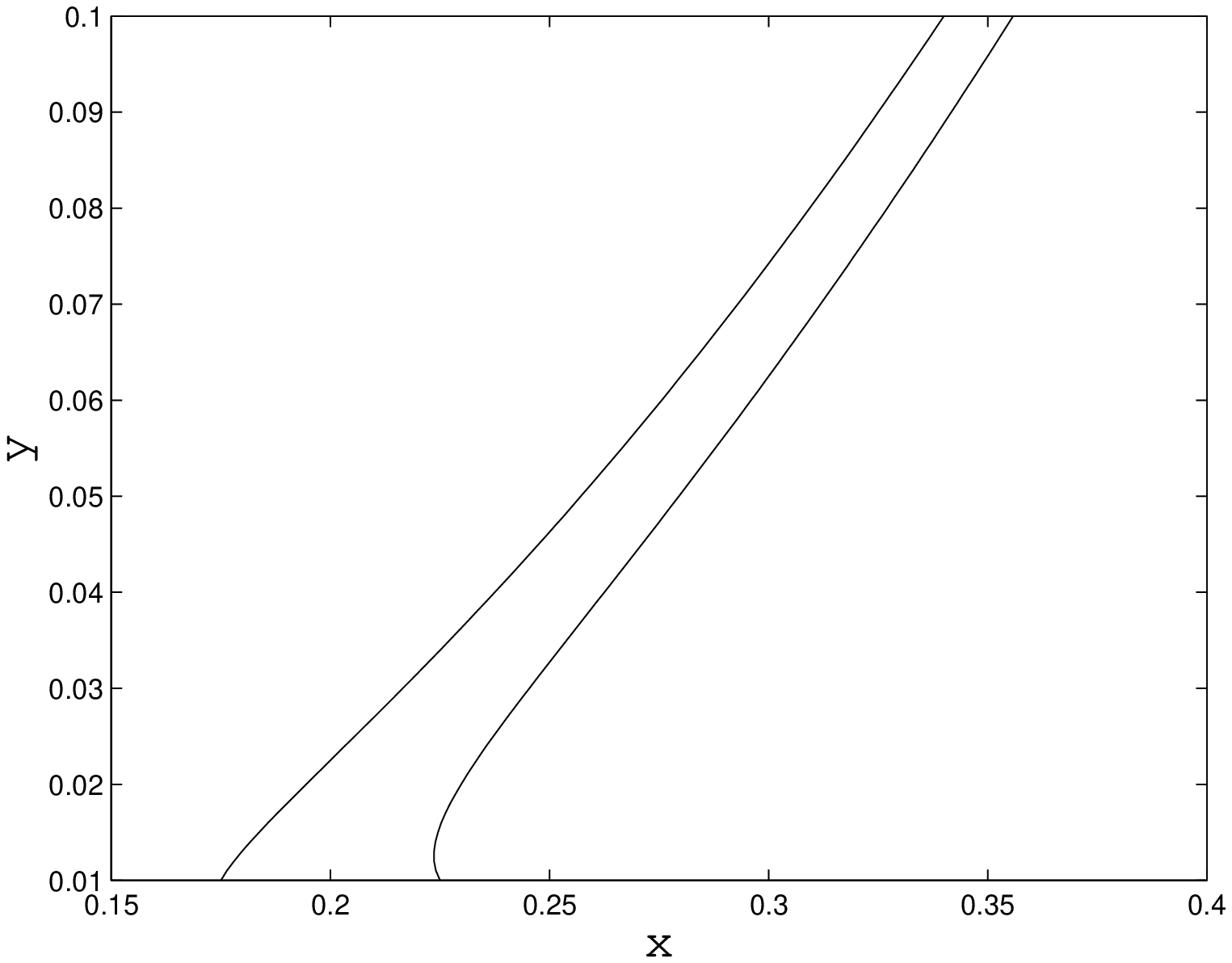}
\caption{Updating the enclosures of the slow manifold. This figure show the projection on $(x,y)$ coordinates of: (a) The initial enclosures. (b) The enclosures are updated vertex wise, here the first half of the vertices are updated. (c) The new enclosure.}\label{F_updateMfd}
\end{center}
\end{figure}

\section{Existence of locally invariant manifolds}\label{S_Existence}
The method outlined in the previous section constructs two triangulated surfaces, in the phase space of a slow-fast system, that are transversal to the flow for the given $\eps$. In this section we discuss the existence of locally invariant manifolds enclosed between these two triangulations. We denote the two enclosing surfaces by $\L$ and $\R$, and the region enclosed between them by $\C$. Note that $\L$ and $\R$ are graphs over the same compact region, so $\C$ is well defined. Specifically, if for some compact set of slow variables $D$, $\L=\{(x,y,z) \,: x=h_l(y,z), (y,z)\in D\}$, $\R=\{(x,y,z) \,: x=h_r(y,z), (y,z)\in D\}$, then  $\C=\{(x,y,z) \,: x\in[h_l(y,z),h_r(y,z)], (y,z)\in D\}$. We would like to claim that there is a locally invariant manifold inside of $\C$ which is a graph over the slow variables. We must thus verify that it is possible to choose a subset of $\C$ which is a $C^1$ manifold, locally invariant, and whose projection onto the domain in the slow variables is bijective. We start this section by defining \textit{computable slow manifolds} as objects associated to a fixed $\eps$. Similar to a slow-manifold, a computable slow manifold, is not unique. Informally, a computable slow manifold is a manifold close to the critical manifold where the flow is slow. We measure slowness by comparing the slopes of trajectories within our enclosure with the slope of the critical manifold. The relative slope is defined as a bound on the slope of trajectories divided by the slope of the critical manifold: 
\begin{equation}\label{eq_relSlopeDef}
s(\eps)=\max_{\C} \dfrac{\dfrac{|\dot x|}{|\dot y|+|\dot z|}}{\left(\left|\dfrac{\partial h_0(y,z)}{\partial y}\right|+\left|\dfrac{\partial h_0(y,z)}{\partial z}\right|\right)}.
\end{equation} 
\begin{definition}
A computable slow manifold is a $C^1$ locally invariant, normally hyperbolic manifold, of the same dimension as the critical manifold, which projects injectively along the fast variable to the critical manifold, such that the its relative slope satisfies $s(\eps)\leq \frac{1}{\sqrt{\eps}}.$ 
\end{definition}
Note that the order of $s(\eps)$ is $O\left(\frac{1}{\eps}\right)$ away from the slow manifold, so the definition is consistent with the standard perturbative definition of slow manifolds \cite{J95}. Slow manifolds are widely used in studies of slow-fast systems arising from biological or chemical models. However, computable slow manifolds are often the objects that are identified in applications: a locally invariant manifold at a \emph{fixed} value of epsilon that follows the critical manifold closely, and on which the flow is slow \cite{DKO08,GK09,I00}. This concept is captured by the definition of the computable slow manifold. Thus, our enclosures method gives a general and robust method to compute where candidates for such manifolds might lie in the phase space.

  
We will explain why computable slow manifolds exist within $\C$ in the following special case, which is sufficient for the purpose of this paper.
\begin{assumption}\label{A_Exist}
Assume that:
\begin{itemize}
 \item[i] All trajectories of $\C$ reach its boundary in forward and backward time.
 \item[ii] The boundary of $\C$, $\partial \C$ is piecewise smooth. Tangencies of the vector field with $\partial \C$ are quadratic (i.e., folds in the sense of singularity theory), and these tangencies occur along smooth curves that connect $\L$ and $\R$.
\item[iii] There are invariant horizontal and vertical cone fields on $\C$, and the vertical invariant cone field contains the fast direction of the vector field on $\C$. 
\end{itemize}
\end{assumption}

Assume that the vector field is inward on $\L$ and $\R$ and denote by $\C_{in}$ and $\C_{out}$ the sets in $\partial \C - \L - \R$ where the vector field points inward and outward, respectively. 
Choose a smooth curve, $x=r_0(y,z)$, in $\partial \C - \L - \R$ such that the projection of the curve to the slow variables contains the projection of $C_{in}$ to the slow variables, and points on the curve on $\C_{out}$ are images of the flow of points on the curve on $\C_{in}$. Flow this graph forward until each trajectory leaves $\C$. The set swept out by these trajectory segments is: 
\begin{equation}\label{eq_hatSeps}
S_\eps:=\{\phi_t(x,y,z) \,: x=r_0(y,z), \phi_t(x,y,z)\in \C\}. 
\end{equation}
The set $S_\eps$ is well-defined, as smooth as $r_0$ and the vector field, and diffeomorphic to its projection onto the critical manifold $S_0$. Inflowing trajectories of $\C$ must exit through $\C_{out}$. The exit time is uniformly bounded, since $\C$ is compact. Hence, $S_\eps$ is well defined. The existence of invariant cone fields, with a normal vertical cone field containing the fast direction, ensures that $S_\eps$ is a graph over the slow variables, and thus diffeomorphic to the corresponding part of the critical manifold. The final requirement of the definition - that the relative slope is small - yields a quantitative requirement on the tightness of $\L$ and $\R$. 
%

\section{Singular Hopf normal form}\label{S_SingHopfBif}

In a slow-fast system, an equilibrium point may cross a fold of the critical manifold. If it undergoes a Hopf bifurcation at $O(\eps)$ distance from the fold both in parameter and phase space, we follow \cite{G08} and refer to this as a \textit{singular Hopf bifurcation}. Singular Hopf bifurcation occurs in generic one parameter families of slow-fast systems. 

We use a normal form for singular Hopf bifurcation in systems with one fast and two slow variables proposed by Guckenheimer \cite{G08} as an example system for the computations of slow manifolds presented in this paper. The normal form is given by
\begin{equation}\label{eq_singularHopfNormalForm}
\begin{split}
\eps\,\dot{x}
& = 
(y-x^2)  \\
\dot{y} & = 
z-x \\
\dot{z}& =
-\mu-a x -b y -c z,\\
\end{split}
\end{equation}
which depends upon the four parameters $\mu, a, b, c$ as well as $\eps$. An $\eps$-dependent scaling transformation eliminates $\eps$ as a parameter: set 
\begin{equation}\label{eq_rescaling}
(X,Y,Z,T) = (\eps^{-1/2}x,\eps^{-1} y,\eps^{-1/2}z,\eps^{-1/2}t){\textrm{ and }}(A,B,C)=(\eps^{1/2}a,\eps b,\eps^{1/2}c) 
\end{equation}
to obtain
\begin{equation}\label{eq_resc_shnf}
\begin{split}
X'
& = 
Y-X^2 \\
Y' & = 
Z-X \\
Z'& =
-\mu-A X -B Y -C Z\\
\end{split}
\end{equation}

\begin{figure}[h]
\begin{center}
\includegraphics[width=0.8 \textwidth]{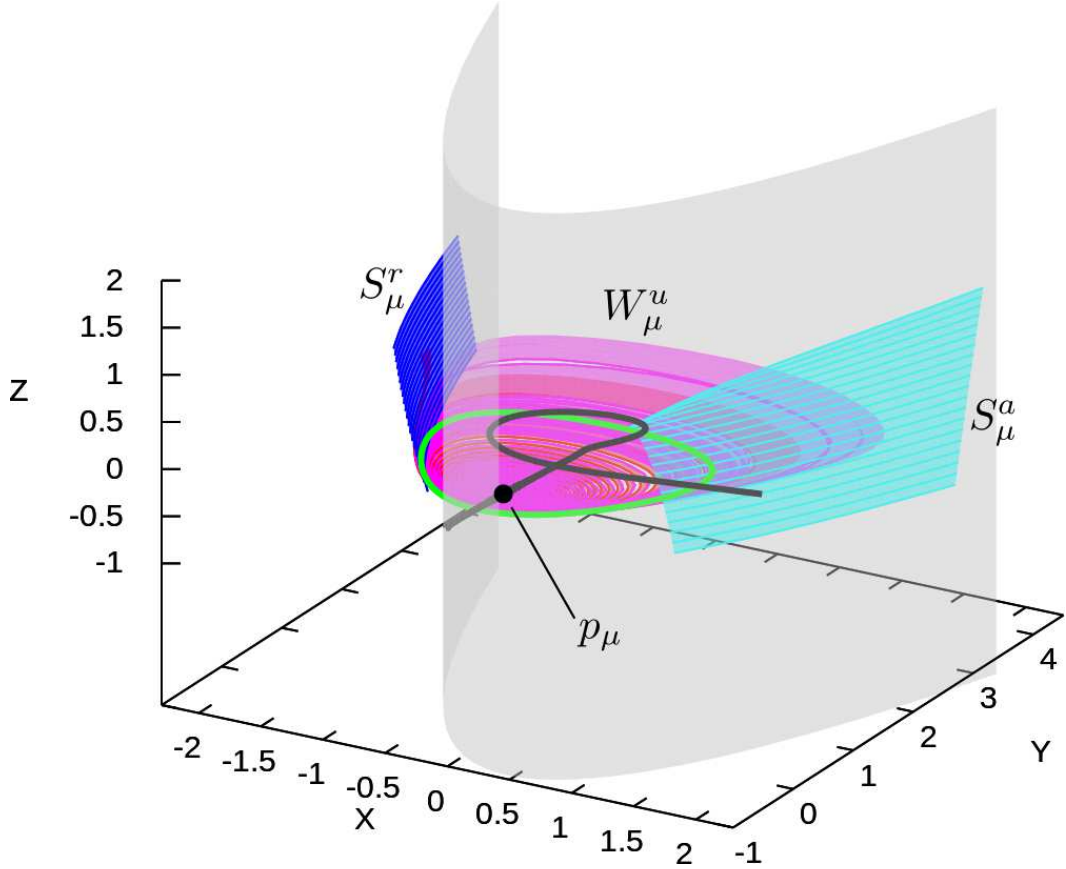}
\caption{Phase space of system ~\eqref{eq_resc_shnf} when $(\mu,A,B,C)=(0.0015709, -0.05, 0.001, 0.1)$. A repelling sheet of the slow manifold is plotted in dark blue, an attracting sheet in cyan. A selection of trajectories in the unstable manifold of the equilibrium near the origin are plotted in red,    a strip of unstable manifold that escapes from the fold region is shaded in magenta.}\label{F_shb_phase_space}
\end{center}
\end{figure}

Guckenheimer \cite{G08} studies invariant manifolds in the phase space of system ~\eqref{eq_resc_shnf} at different system parameters: 
The branch of the critical manifold $y=x^2$ where $x<0$ perturbs into a repelling slow manifold, while the branch where $x>0$ perturbs into an attracting slow manifold. 
In large regions of parameter space, an equilibrium that has undergone singular Hopf bifurcation is a saddle-focus with a two-dimensional unstable manifold that is initially bounded by a periodic orbit. As the parameter $\mu$ is varied, the unstable manifold grows and eventually intersects the repelling slow manifold, first tangentially and then transversally. 
In the following, we refer to such a tangential intersection of the equilibrium's unstable manifold with the repelling slow manifold as a \textit{tangency} or \textit{tangency of invariant manifolds}. Figure ~\ref{F_shb_phase_space} shows selected objects in the phase space of system ~\eqref{eq_resc_shnf} at a parameter just after the tangency. The tangency is a codimension 1 phenomenon. Note that since slow manifolds are not unique, there is some ambiguity in what it means for a slow manifold to intersect another manifold. We will introduce a definition to deal with this in section \ref{S_Tang}.

The tangency bifurcation is evident in the organization of phase space by invariant manifolds, as it  separates regions in parameter space where trajectories in the unstable manifold of the singular Hopf equilibrium have different possible limit sets: after the tangency, trajectories can escape from the fold region, whereas, before the tangency, the unstable manifold is confined to the fold region. This change is significant in many other slow-fast systems: suppose that a system with one fast and two slow variables has an S-shaped critical manifold as well as a singular Hopf bifurcation  followed by a tangency. The system introduced by Koper  \cite{Koper, KrupaPopovicKopell, Kuehn, DGKKO} is an example. Figure ~\ref{F_Koper_phasespace_timeseries} (left pane) shows a trajectory in Koper's system just after the tangency bifurcation, together with the position of the critical manifold: the trajectory starts in the vicinity of the singular Hopf equilibrium and goes through a spiraling motion until it leaves the fold region to the left. It lands close to the critical manifold on an attracting slow manifold, which it follows to a fold before ``jumping'' to another attracting slow manifold and returning along this slow manifold back to the vicinity of the singular Hopf equilibrium point. The process now repeats, leading to a time series like the one shown on the right pane of Figure ~\ref{F_Koper_phasespace_timeseries}. Such patterns with alternating large-amplitude and small-amplitude oscillations are known as mixed mode oscillations in the literature ~\cite{DGKKO}.

\begin{figure}[h]
\begin{center}
\includegraphics[width=1.0 \textwidth]{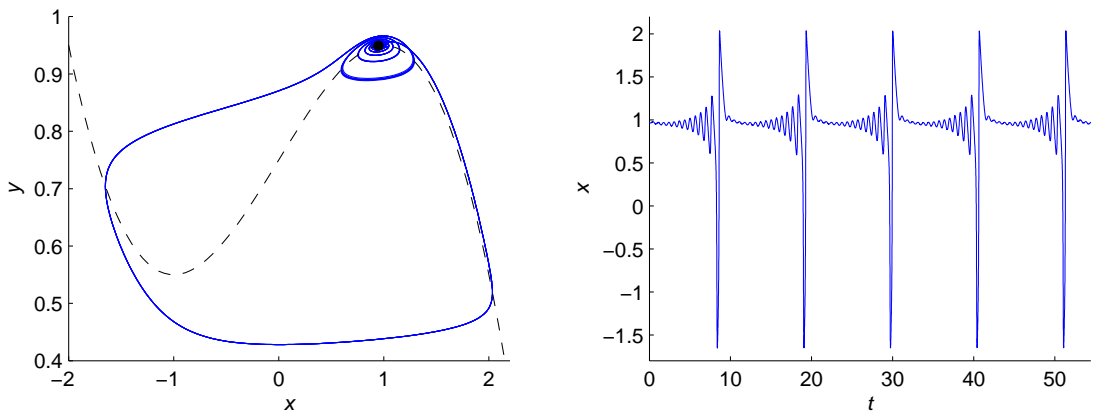}
\caption{Mixed mode oscillations in Koper's system with system parameters $\eps_1=0.1,\eps_2 = 1, k=-10, \lambda = -7.50$: the left panel shows a trajectory in the $xy$-plane, where $x$ is a fast variable while $y$ is slow.  The singular Hopf equilibrium is marked with a black dot, the critical manifold is drawn with a dashed black line. The right panel shows the time-series of the $x$-coordinate of the same trajectory. }\label{F_Koper_phasespace_timeseries}
\end{center}
\end{figure}

Guckenheimer and Meerkamp \cite{GM11} present a detailed analysis of the bifurcation structure of the singular Hopf normal form. This includes extensive numerical results on the position of the tangency bifurcation in the five-dimensional parameter space of the normal form. The position of the tangency curve was computed in two ways: (1) using the numerical continuation software AUTO \cite{AUTO} and (2) using custom MATLAB code. 
In method (1), a boundary value problem is set up to track a trajectory segment that starts on a fixed ray in the linear approximation of the unstable manifold of the singular Hopf equilibrium point and follows the repelling slow manifold for a substantial period of time. The latter is achieved by requiring it to have a sufficiently long time-length and ending on the parabola $y=x^2+5$, thus forcing the trajectory to remain very close to the repelling slow manifold nearly to the end of the trajectory segment. 
A tangency of invariant manifolds corresponds to a fold of a solution to this boundary value problem, i.e., a point where two solutions approach each other and vanish together as the parameter is varied. Such folds can be detected  by AUTO in a one-parameter continuation and continued in two parameters, thus enabling detection and continuation of the tangency bifurcation. Method (2) is based on the heuristic that the tangency separates regions in parameter space where trajectories in the unstable manifold escape the fold region from regions where trajectories in the unstable manifold remain in the fold region. A grid of initial conditions in an approximation  of a fundamental domain of the unstable manifold is integrated numerically for a sufficiently long time. If the sample trajectories in the unstable manifold limit to an attracting periodic orbit or another bounded attractor, the parameter is before the tangency. If at least one trajectory leaves the fold region, the parameter is after the tangency. Repeated applications of the above steps in an interval bisection method determine an approximate position of the tangency in parameter space. Note that neither method (1) nor (2) is rigorous. In particular, neither of the two methods establishes bounds for the position of the repelling slow manifold or of the unstable manifold of the equilibrium. Section ~\ref{S_Tang} presents a rigorous method to compute a tangency.


\section{Detailed description of the method for the singular Hopf normal form}\label{S_Method}
In this section we give a detailed description of our method for computing enclosures of slow manifolds, applying it to the system from Section \ref{S_SingHopfBif} as an example. Most of the details generalize to any system of the form (\ref{eq_slowFast_12}). In the description, we comment on nontrivial differences between the general case and the example at hand.

\subsection{Constructing the triangulation}\label{SS_Triangulation}
The first step of our algorithm is to triangulate a portion of the critical manifold $S_0$. On a normally hyperbolic piece of the critical manifold, $\partial_x f \neq 0$. The implicit function theorem implies there is locally a function $h_0(y,z)$, such that $f(h_0(y,z),y,z)=0$. In the singular Hopf normal form, $h_0$ is given explicitly as $h_0^\pm(y,z)=\pm \sqrt{y}$ with domain $D = [y_m,y_M]\times [z_m,z_M]\subset\RR^2$. For other systems, any 
suitable method for finding a sufficiently accurate approximation to $h_0(y,z)$ can be used.

To construct the vertices of a Delaunay triangulation of $S_0$, as shown in Figure \ref{F_triangulation}(a), we start with a 
triangulation of the domain of $h_0$, but want the diameter of the triangles on $S_0$ to be almost uniform. Setting $\kappa(y,z)=\|\nabla h\|$, $\tilde k=\sqrt{(y_M-y_m)^2+(z_M-z_m)^2}/d$, and
$$k(y,z):=\frac{\tilde k}{1+\kappa(y,z)},$$
with $d\in\Z_+$ to be chosen later, we select the following points in the $(y,z)$ plane as vertices of a triangulation: 
\begin{eqnarray}\label{eq_vertexPts}
 \nonumber(y_0,z_0) :=& (y_m,z_m) \\
 \nonumber(y_i,z_0) :=& (y_{i-1}+k(y_{i-1},z_0),z_0) ,\quad  & \textrm{if} \quad y_{i-1}<y_{i-1}+k(y_{i-1},z_0)<y_M, \\
 (y_i,z_0) :=& (y_M,z_m) ,\quad & \textrm{if} \quad y_{i-1}<y_M\leq y_{i-1}+k(y_{i-1},z_0), \\
 \nonumber(y_i,z_j) :=& (y_i,z_{j-1}+k(y_i,z_{j-1})) ,\quad  & \textrm{if} \quad z_{j-1}<z_{j-1}+k(y_i,z_{j-1})<z_M, \\
 \nonumber(y_i,z_j) :=& (y_i,z_M) ,\quad & \textrm{if} \quad z_{j-1}<z_M\leq z_{j-1}+k(y_i,z_{j-1}), \\
 & 0\le i \le I, \, 0\le j(i)\le J_i
\end{eqnarray}
Note that these points are aligned along lines parallel to the fold curve $x=y=0$ where $\partial_x f=0$. 

Let $\T$ denote the Delaunay triangulation generated by the set 
$$
\{(y_i,z_i) :  0\le i \le I, 0\le j(i)\le J_i \},
$$
 and $\K_0$ its lift to $S_0$, using the map $\pi_0^{-1}: (y,z) \mapsto (h_0(y,z),y,z)$.
Clearly $\pi_0^{-1}$ is a homeomorphism; i.e., the set of vertices, edges, and faces of $\K_0$, denoted by $V(\K_0)$, $E(\K_0)$, and $F(\K_0)$, are defined by $\pi_0^{-1}(V(\T))$, $\pi_0^{-1}(E(\T))$, and $\pi_0^{-1}(F(\T))$, respectively. $\T$ and $\K_0$ are shown in Figures \ref{F_triangulation}(a) and \ref{F_triangulation}(b), respectively.

\subsection{Constructing perturbed triangulations}
Our next step is to perturb $\K_0$, as illustrated in Figure \ref{F_critSlowMfd}, so that it lies closer to the slow manifold $S_\eps$
we are trying to enclose.
Fenichel theory, \cite{J95}, guarantees that for $\eps>0$ sufficiently small, $S_\eps$ is the graph of a function $h_\eps(y,z)$ with domain $D$ and $h_\eps(y,z) - h_0(y,z) = O(\eps)$. To compute triangulations, $\K_\eps$ that approximate $S_\eps$, we write $h_\eps$ in the form 
$$
h_\eps(y,z)=h_0(y,z)+\eps h_1(y,z).
$$
Substituting into the equation $\eps \dot x_\eps=f(h_\eps(y,z),y,z)$, we get that:
\begin{eqnarray}\label{eq_h1AssFor}
\nonumber f(h_0(y,z)+\eps h_1(y,z)+O(\eps^2),y,z)/\eps & = & \partial_y (h_0(y,z)+\eps h_1(y,z))\dot y+\partial_z(h_0(y,z)+\eps h_1(y,z))\dot z \\
\nonumber  & = & \partial_y h_0(y,z)\dot y+\partial_z h_0(y,z)\dot z+O(\eps) \\
\nonumber & = & \partial_y h_0(y,z)g_y(h_0(y,z),y,z) \\ 
& &+\partial_z h_0(y,z)g_z(h_0(y,z),y,z)+O(\eps)
\end{eqnarray}
To compute $\partial_y h_0$ and $\partial_z h_0$, we use that $f(h_0(y,z),y,z)=0$, and hence
$$
\partial_y h_0(y,z) = -\frac{\partial_y f(h_0(y,z),y,z)}{\partial_x f(h_0(y,z),y,z)},
$$
and
$$
\partial_z h_0(y,z) = -\frac{\partial_z f(h_0(y,z),y,z)}{\partial_x f(h_0(y,z),y,z)}.
$$
In addition, since $f(h_0(y,z),y,z)=0$, 
\begin{equation}\label{eq_fAssFor}
f(h_\eps(y,z),y,z)=\eps\partial_x f(h_0(y,z),y,z) h_1(y,z)+O(\eps^2).
\end{equation}
Thus, we can solve equation \eqref{eq_fAssFor} for $h_1(y,z)$, up to $O(\eps)$, and substitute for $f(h_\eps(y,z),y,z)$ using \eqref{eq_h1AssFor}, obtaining
$$
h_1(y,z) = -\frac{\partial_y f(h_0(y,z),y,z)g_y(h_0(y,z),y,z)+\partial_z f(h_0(y,z),y,z)g_z(h_0(y,z),y,z)}{\left(\partial_x f(h_0(y,z),y,z)\right)^2}+O(\eps),
$$
which in our case, considering $h^+(y,z)$ reads:
\begin{equation}\label{eq_h1Def}
h_1^+(y,z) = \frac{\sqrt{y}-z}{4y}.
\end{equation}
For $h^-(y,z)$, that we will use in Section \ref{S_Tang}, we get:
\begin{equation}\label{eq_h1DefM}
h_1^-(y,z) = \frac{-\sqrt{y}-z}{4y}.
\end{equation}

We put $\pi_\eps^{-1}: (y,z) \mapsto (h_0(y,z)+\eps h_1(y,z),y,z)$, and define:
\begin{equation*}
\K_\eps := \pi^{-1}_\eps \circ \pi_0(\K_0).
\end{equation*}
$\K_\eps$ is our approximation to the slow manifold, shown together with $S_0$ in Figure \ref{F_critSlowMfd}(b). Heuristically, it is $O(\eps^2)$ to $S_\eps$ at the vertex points. 

Let $\sigma_c$ denote the following map that moves points parallel to the $x$-axis:

\begin{equation}\label{eq_sigmaDef}
\sigma_c: (x,y,z) \mapsto  (x+c\max\left(|h_1(y,z)|,\frac{\epsilon^2}{|c|}\right),y,z). 
\end{equation}

We define our candidate enclosing surfaces as:
\begin{eqnarray}
 \L_{\eps, N} & := & \sigma_{-\eps/N}(\K_\eps) \\
 \R_{\eps, N} & := & \sigma_{\eps/N}(\K_\eps),
\end{eqnarray}
where $N\in\RR_+$. The initial choice for $N$ in our implementation was $N=64$, but we would have chosen a smaller $N$ if that had failed. The verification step of the algorithm includes a loop that divides $N$ by a factor $2$ upon failure and repeats the transversality test. Note that the region that is enclosed by $\L_{\eps, N}$ and $\R_{\eps, N}$ is disjoint from the critical manifold so long as $N>1$. The construction of $S_0$, $S_\eps$, $\L_{\eps, N}$ and $\R_{\eps, N}$ is shown in Figure \ref{F_critSlowMfd}.

\subsection{Verifying the enclosure property}
To prove that a slow manifold is located between $\L_{\eps,N}$ and $\R_{\eps,N}$, it suffices to prove that the vector field $(\ref{eq_slowFast_12})$ is transversal to each face of the triangulations, with opposite crossing directions for $\L_{\eps,N}$ and $\R_{\eps,N}$. For the remainder of this subsection, we restrict our attention to a single triangle. Local transversality, i.e., the verification of transversality on each face in the triangulation implies global transversality of $\L_{\eps,N}$ and $\R_{\eps,N}$.

Let $T$ be one face in $\L_{\eps,N}$ or $\R_{\eps,N}$. We denote its vertices by $v_1$, $v_2$, and $v_3$ and its edges by $e_{12}$, $e_{13}$, and $e_{23}$ with the edge $e_{ij}$ between the vertices $v_i$ and $v_j$. To verify that the vector field is transverse, it suffices to prove that the inner product between the normal of the face and the vector field is non-zero. Note that in contrast to most work on slow-fast systems, this condition, which is the main condition checked by our algorithm, becomes \textit{easier} to verify as $\eps\rightarrow 0$. The reason is that as $\eps\rightarrow 0$, the condition becomes essentially one-dimensional. We denote the normal to the face, normalized so that the first component is positive, by $n(T)$. This is possible because the first component is zero exactly at the folds, where the critical manifold fails to be normally hyperbolic. With this notation, the condition that we have to verify is 
\begin{equation}\label{eq_transCond}
 F(x,y,z)\cdot n(T) \neq 0, \quad \textrm{for all } (x,y,z)\in T.
\end{equation}

Condition (\ref{eq_transCond}) is equivalent to a verification that
\begin{equation}\label{eq_transCondFace}
F(\lambda_1 v_1+\lambda_2 v_2+\lambda_3 v_3)\cdot n(T) \neq 0 \quad \textrm{for all } \lambda_i\in[0,1], \lambda_1+\lambda_2+\lambda_3=1,
\end{equation}
which is an enclosure of the range of a function on a compact domain. This problem is the one we solve with interval analysis. Directly enclosing (\ref{eq_transCondFace}) using interval analysis in order to verify that the function is non-zero is, however, not optimal. The reason is that the problem is sufficiently sensitive that we would have to split the $\lambda_i$ domains into a very fine subdivision, and since this has to be done on each face, such a procedure would be prohibitively slow.

Our actual approach is based on monotonicity; first we prove that $F\cdot n$ is monotone on the face and on its restriction to the edges. Then we compute $F(v_i)\cdot n$ for the three vertices and verify that the interval hull of the results, i.e., the smallest representable interval containing the results, does not contain $0$. Note that this amounts to showing that the dot-product does not change sign on the face. We introduce
\begin{equation}\label{eq_SDef}
G := \nabla (F\cdot n).
\end{equation}
If $G\neq (0,0,0)$ on all of $T$ then $F\cdot n$ has no critical points inside of $T$ and we can restrict our attention to the edges, i.e. the boundary of $T$. Consider an edge $e_{ij}=\{(1-\lambda)v_i+\lambda v_j \,:\,\lambda\in[0,1] \}$, and denote its parametrization by $r(\lambda)$. The scalar product $F\cdot n$ is monotone on the edge if 
$$
0 \neq \frac{\partial}{\partial \lambda} (F(r(\lambda))\cdot n) =
G\cdot (v_j-v_i). 
$$  

Hence, we arrive at the monotonicity requirements, which for the case at hand are much easier to verify than \eqref{eq_transCondFace}:
\begin{eqnarray}
(0,0,0) &\notin& G(T) \label{eq_transEq1}\\
0 &\notin& G(e_{12}) \cdot (v_2-v_1) \label{eq_transEq2}\\
0 &\notin& G(e_{13}) \cdot (v_3-v_1) \label{eq_transEq3}\\
0 &\notin& G(e_{23}) \cdot (v_3-v_2) \label{eq_transEq4}
\end{eqnarray}

If the conditions (\ref{eq_transEq1}-\ref{eq_transEq4}) are satisfied we compute 
\begin{equation}\label{eq_vertexTrans}
F(v_1)\cdot n \sqcup F(v_2)\cdot n \sqcup F(v_3)\cdot n, 
\end{equation}
where $\sqcup$ denotes the interval hull. If (\ref{eq_vertexTrans}) does not contain zero, then the vector field is transversal to the face $T$. If (\ref{eq_transEq1}) holds but one or more of (\ref{eq_transEq2}-\ref{eq_transEq4}) do not hold, then we add the appropriate $F(e_{ij})\cdot n$ terms to (\ref{eq_vertexTrans}).

\subsection{Improving the bounds}\label{SS_ImpBound}
If the previous steps of the algorithm are successful, they yield two surfaces $\L_{\eps,N}$ and $\R_{\eps,N}$, that have been proven to enclose the part of the slow manifold that is above $[y_m,y_M]\times[z_m,z_M]$ in the $(y,z)$ plane. Since $N$ is fixed after the verification step we henceforth drop the indices on $\L$ and $\R$. Our aim is to produce enclosures that are as tight as possible, given the mesh size. We, therefore, try to improve the enclosure. The procedure is illustrated in Figure \ref{F_updateMfd}.

We do this by iteratively updating each of the vertices in the triangulation by moving them towards each other along the segment joining them. This segment is parallel to the x-axis due to our earlier constructions. The moves are done in two steps: (1) a tentative move is made of a vertex, and (2) the transversality conditions of all faces attached to this vertex are verified. When the transversality holds, the vertex is fixed at its new position and we proceed to the next vertex. The efficiency of this procedure will depend on several factors, primarily the ordering of the vertices and how much the vertices are moved. By moving a vertex only a fraction of what seems to be possible, the effect of the ordering of the vertices can be minimized. The penalty of smaller updates is that the procedure has to be run more times. Larger moves might be possible if an appropriate sorting algorithm were used, but we have not found an effective and efficient sorting criterion. Instead, we heuristically determine an update factor that optimizes the accuracy vs complexity. Given a right vertex, $v_R$, and a left vertex, $v_L$, such that $\pi_0(v_R)=\pi_0(v_L)$, we move each of them towards each other by an amount 
\begin{equation}\label{eq_update}
\frac{1}{8}\|(v_R-v_L)\|.
\end{equation}

We run the procedure to refine the enclosures of the slow-manifold several times, until no further improvement is possible. The quantity we use to measure the quality of the enclosures is the average distance between the two triangulations at the vertices. Let $\iota$ denote the number of vertices of the triangulations; by construction $\L$ and $\R$ have the same number of vertices, edges, and faces. The only difference between $\L$ and $\R$ is the values of the $x$-coordinates. We put
\begin{equation}\label{D_eta}
\eta(\L,\R)= \frac{1}{\sqrt{\iota}} \|v_{R}-v_{L}\|. 
\end{equation}
If the triangulation is fine enough $\eta$ will be $O(\eps^2)$. This fact is investigated numerically in Section \ref{S_NumRes}.

\subsection{Cone fields}
In order to ensure that there are manifolds inside of the set $\C$ enclosed by $\L$ and $\R$, we need to have invariant cone fields on $\C$, as introduced in Section \ref{S_Existence}. In this subsection we describe how such cone fields - one horizontal and one vertical - are constructed. Recall, see \cite{KH95}, that a standard horizontal or vertical cone for a phase space with variables $(x,y)$ is a set $\{\gamma\|x\|\geq \|y\|\}$ or $\{\gamma\|y\|\geq \|x\|\}$, respectively, and that a cone is the image of a standard cone under an invertible linear map. Equivalently, a cone is the set of points where a non-degenerate indefinite quadratic form is non-negative. Since horizontal and vertical cones are traditionally in the expanding and contracting directions, respectively, we will call the cone in the normal direction the vertical cone, and the cone in the direction of the slow manifold the horizontal cone. Also recall that a cone field is invariant if it is mapped into itself by the derivative of the dynamics, i.e., if the set where the quadratic form is non-negative is mapped by the derivative into the set where the quadratic form at the image point under the map is non-negative.

For the case at hand we will use $\gamma=1$ for both the horizontal and vertical cones in an appropriate coordinate system, such that the normal direction is in the vertical cone. A cone field is a map that associates a cone to each point of its domain. Given that \eqref{eq_resc_shnf} only has one nonlinear component, we will use constant cone fields. To prove that the cone fields are invariant, we solve the variational equation for the time $0.0004$ flow map, and use the eigendirections of the derivative of the flow as a basis, in which we represent the standard horizontal and vertical cones with $\gamma=1$. We verify that the vertical and horizontal cone fields are invariant, and that the vertical cone contains the fast direction, which ensures that $\hat S_\eps$ defined in \eqref{eq_hatSeps} projects injectively onto the slow variables, and, thus, is a graph over them. The flow time needs to be large enough for us to be able to prove the separation of the horizontal and vertical directions, but small enough that we do not move away too far in phase space. The value $0.0004$ turned out to be a good choice.

\subsection{Algorithms}\label{SS_Algorithm}
An implementation \cite{Progs} of the method described above has been made using the IntLab package \cite{IntLab} for interval arithmetic. A detailed description of the main algorithm is given as Algorithm \ref{mainAlgorithm}. The algorithm that checks if the vector field is transversal to a face is given as Algorithm \ref{transversalityAlgorithm}. Algorithm \ref{mainAlgorithm} takes a triangulation as input. That triangulation can be computed with any method, not necessarily the one outlined in Section \ref{SS_Triangulation}. In Algorithm \ref{transversalityAlgorithm} the function $sign(x)$ returns $0$ if $0\in x$. 


\begin{small}
\begin{algorithm}[ph]\label{mainAlgorithm}
 \KwData{$(f,g_y,g_z)$, $h_0$, $\T$, $\eps$}
\KwResult{$\L$, $\R$, $\eta$}
\ForAll{$(y,z)\in \T$}{$h_1(y,z) = -\frac{\partial_y f(h_0(y,z),y,z)g_y(h_0(y,z),y,z)+\partial_z f(h_0(y,z),y,z)g_z(h_0(y,z),y,z)}{\left(\partial_x f(h_0(y,z),y,z)\right)^2}$\;
}
$N=64$\;
transversal=false\;
$NF=\T.numberOfFaces$\;
\While{$\neg transversal$ \& $N>2^{-18}$}{
$x_{left}=h_0(y,z)+h_1(y,z)-\eps/N |h_1(y,z)|$\;
$x_{right}=h_0(y,z)+h_1(y,z)+\eps/N |h_1(y,z)|$\;
\eIf{$getTransversality(\T,x_{left})=-getTransversality(\T,x_{right})=NF$}
{transversal=true\;}
{$N=N/2$\;}
}
\If{$\neg transversal$} 
{exit(FAIL)\;}
$\eta=1$\;
$\eta_{new}=0$\;
\While{$\eta_{new}<\eta$}{
$\eta=\frac{\|x_{left}-x_{right}\|}{\sqrt{T.\iota}}$\;
$\tilde x_{left}=x_{left}$, $\tilde x_{right}=x_{right}$\;
\ForAll{$1\leq i\leq \iota$}{
$tri=\T.adjacentFaces(i)$\;
$\tilde x_{left}(i)=x_{left}(i)+0.125(x_{right}(i)-x_{left}(i))$\;
\eIf{$getTransversality(tri,\tilde x_{left},T.y,T.z)=-getTransversality(tri,x_{right},T.y,T.z)=tri.numberOfFaces$}{$x_{left}(i)=\tilde x_{left}(i)$\;}{$\tilde x_{left}(i)=x_{left}(i)$\;}
$\tilde x_{right}(i)=x_{right}(i)-0.125(x_{right}(i)-x_{left}(i))$\;
\eIf{$getTransversality(tri,x_{left},T.y,T.z)=-getTransversality(tri,\tilde x_{right},T.y,T.z)=tri.numberOfFaces$}{$x_{right}(i)=\tilde x_{right}(i)$\;}{$\tilde x_{right}(i)=x_{right}(i)$\;}
}
$\eta_{new}=\frac{\|x_{left}-x_{right}\|}{\sqrt{T.\iota}}$\;
}
$\L=Triangulate(\T.Triangulation,x_{left},\T.y,\T.z)$\;
$\R=Triangulate(\T.Triangulation,x_{right},\T.y,\T.z)$\;
\caption{Implementation of the main algorithm} 
\end{algorithm}
\end{small}

\begin{small}
\begin{algorithm}[h]\label{transversalityAlgorithm}
\KwData{$F=(f,g_y,g_z)$, $\T$(Triangulation,Vertices)}
\KwResult{$Intersections$}
$NF=\T.numberOfFaces$\;
$Intersections=0$\;
\ForAll{$1\leq i\leq NF$}{
$n=\T.Normal(i)$\;
$(v_{1},v_{2},v_{3})=\T.Vertices(i)$\;
$(e_{12},e_{13},e_{23})=\T.Edges(i)$\;
$G=\nabla(F(\T.Face(i))\cdot n)$\;
\eIf{$0\in G$}
{$Intersections+=sign(F(\T.Face(i))\cdot n)$\;}
{$G_{12}=\nabla(F(e_{12})\cdot n)\cdot e_{12}$,
$G_{13}=\nabla(F(e_{13})\cdot n)\cdot e_{13}$,
$G_{23}=\nabla(F(e_{23})\cdot n)\cdot e_{23}$\;
\eIf{$0\notin G_{12}G_{13}G_{23}$}{
$Intersections+=sign(F(v_1)\cdot n \sqcup F(v_2)\cdot n \sqcup F(v_3)\cdot n)$\;
}{
\ForAll{$a\in\{12,13,23\}$}{
\eIf{$0\in G_a$}
{$F_a=F(e_a)\cdot n$}
{$F_a=F(v_{a_1})\cdot n \sqcup F(v_{a_2})\cdot n$}
}
$Intersections+=sign(F_{12}\sqcup F_{13}\sqcup F_{23})$\;
}
}
}
\caption{getTransversality(Triangulation,Vertices)} 
\end{algorithm}
\end{small}

%

\section{Numerical Results}\label{S_NumRes}
In this section we describe the results of several experiments illustrating the behavior of the enclosure computations. Given a system and a domain, there are two numbers that can be changed, the number $d$, which controls the mesh size, and the value of $\eps$. In the experiments below, we use the normal form, \eqref{eq_singularHopfNormalForm}, for the singular Hopf bifurcation discussed in Section 3. We choose the same values of the constants as in the first part of \cite{GM11}: $\mu=10^{-2}$, $A=-0.05$, $B=0.001$, and $C=0.1$. We enclose the branch of the critical manifold $\{y=x^2\}$ with $x>0$. The results of four experiments are described below, in each of them we present the results as a plot of $\eta$ vs $\eps$. In the first experiment, we fix the domain as a small strip: $y\in[0.01,0.2]$, $z\in[-0.01,0.01]$ and give the results for several values of $\iota$ (defined implicitly by changing $d$). In the second, we take a square domain: $y\in[0.01,0.2]$, $z\in[-0.095,0.095]$ for comparison. Our third example analyzes the effect and usefulness of the tightening step described in Section \ref{SS_Tight}. In our fourth example, we investigate the heuristic constant $8$ in the denominator of (\ref{eq_update}); the domain and constants are from the first example with its finest mesh. Note that our domains are such that $\dot y <0$, which means that the assumptions from Section \ref{S_Existence} are satisfied, i.e., all trajectories with initial conditions in $\C$ leave in both forward and backward time, and tangencies of the vector field with $\partial \C$ occur along a plane where they have quadratic tangency. 

During the computations we use the function $G$ defined in (\ref{eq_SDef}) to prove the monotonicity properties that enables us to efficiently prove transversality. We note that for the example at hand, $G$ is
\begin{center}
$
\left(\begin{array}{c}
       -\frac{2x}{\eps}n_x-n_y+\frac{0.05}{\sqrt{\eps}}n_z\\
\frac{1}{\eps}n_x-\frac{0.001}{\eps} n_z \\
n_y-\frac{0.1}{\sqrt{\eps}}n_z
      \end{array}\right).
$ 
\end{center}
A trivial calculation shows that $G=(0,0,0)$ if and only if $x=-25\sqrt{\eps}$ and $n$ is a multiple of $(1,\frac{100}{\sqrt{\eps}},1000)$, so monotonicity always holds on the right branch of the critical manifold.

\subsection{Varying $\iota$}\label{SS_Iota}
The convergence rate of the enclosures at the vertex points should ideally be $O(\eps^2)$, since we have corrected for the linear term in the asymptotic expansion of $h_\eps$. Our interpolating surfaces between the vertex points are, however, linear. The discretization size thus puts a curvature dependent restriction on the tightness of the enclosure. 
In Figure \ref{F_varyIota}(a), we illustrate how $\eta$, for different values of $\iota$ first decreases, but then reaches a plateau. Looking at $\eta$ as a function of $\eps$, we see that as the mesh size decreases ($\iota$ increases), $\eta$ is approximately proportional to $\eps^2$, as expected. This gives a heuristic picture of how $\eta$ depends on $\eps$: first, there will be a period of quadratic convergence, where the accuracy depends on $\eps$; while at the end, the accuracy oscillates around some fixed value and depends on the mesh size. In the intermediate region, the accuracy depends both on the ratio of time scales and the mesh size. In this region, the exponent will decrease from $2$ to $0$. Figure \ref{F_varyIota}(b) illustrates the quadratic convergence region for the finest mesh size from Figure \ref{F_varyIota}(a).

As the plateau is reached $s(\eps)$ defined in \eqref{eq_relSlopeDef} starts to increase. For $\eps=0.1$ the enclosure is too wide for all trajectories inside to be slow. In Table \ref{T_varyIotaSlopes} we give the slopes on the $\eps$ interval $[10^{-1},10^{-4}]$ and bounds on the intervals where $\sqrt(\eps)s(\eps)\leq 1$, for the various $\iota$ values from Figure \ref{F_varyIota}(a). We are only able to prove that the cone fields are invariant for $\eps\leq 10^{-1.94}$, which means that for $\eps>10^{-1.94}$ the normal hyperbolicity is too weak for the algorithm to work. Thus, for the finest mesh size, we prove that the computable slow manifold exists for $10^{-6}\leq\eps\leq10^{-1.94}$. Finer meshes would prove the existence for smaller values of $\eps$.

\begin{figure}[h]
\begin{center}
\psfrag{E}{$-\log_{10}\eps$}
\psfrag{N}{$\log_{10}\eta$}
(a)\includegraphics[width=0.45 \textwidth]{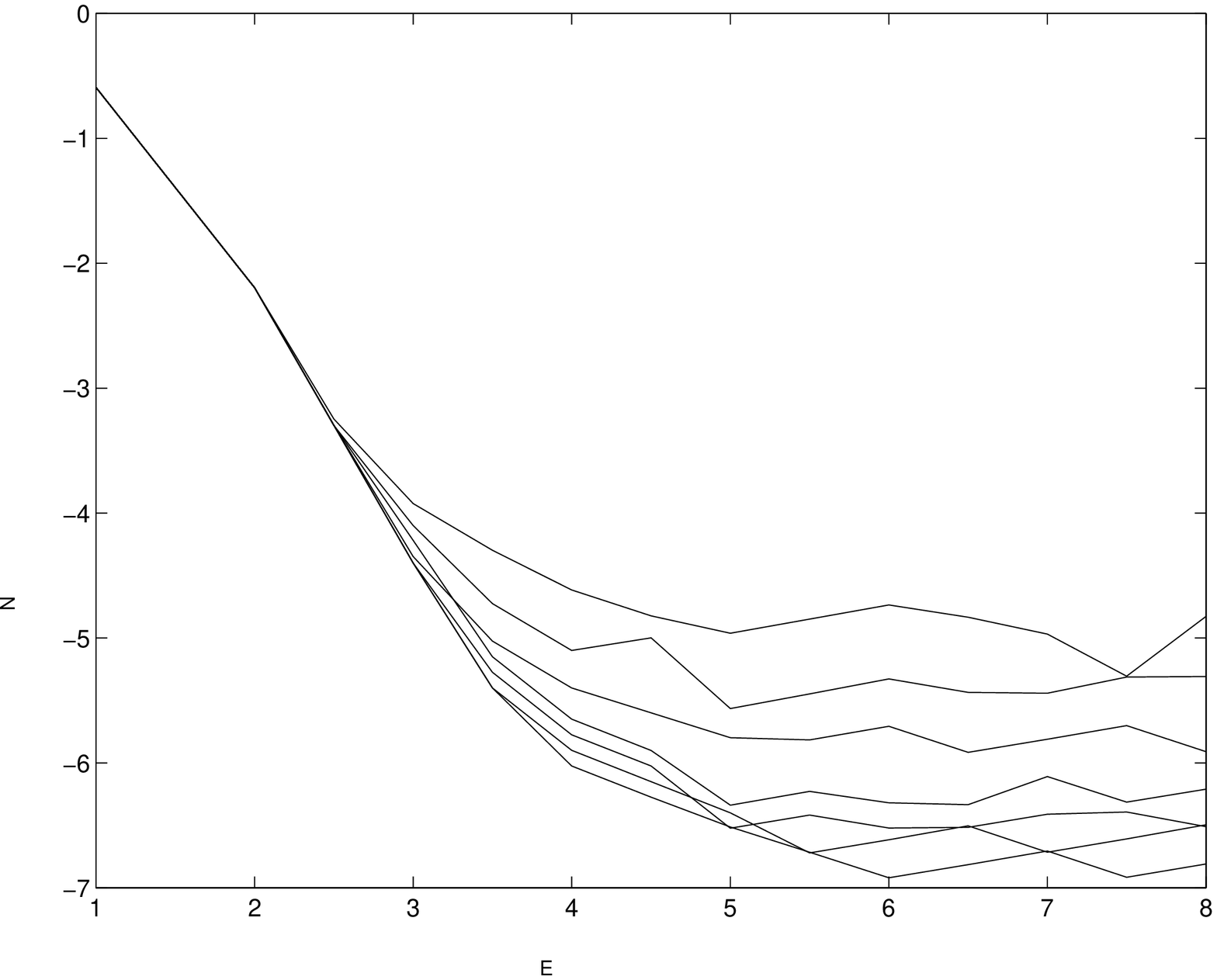}
(b)\includegraphics[width=0.45 \textwidth]{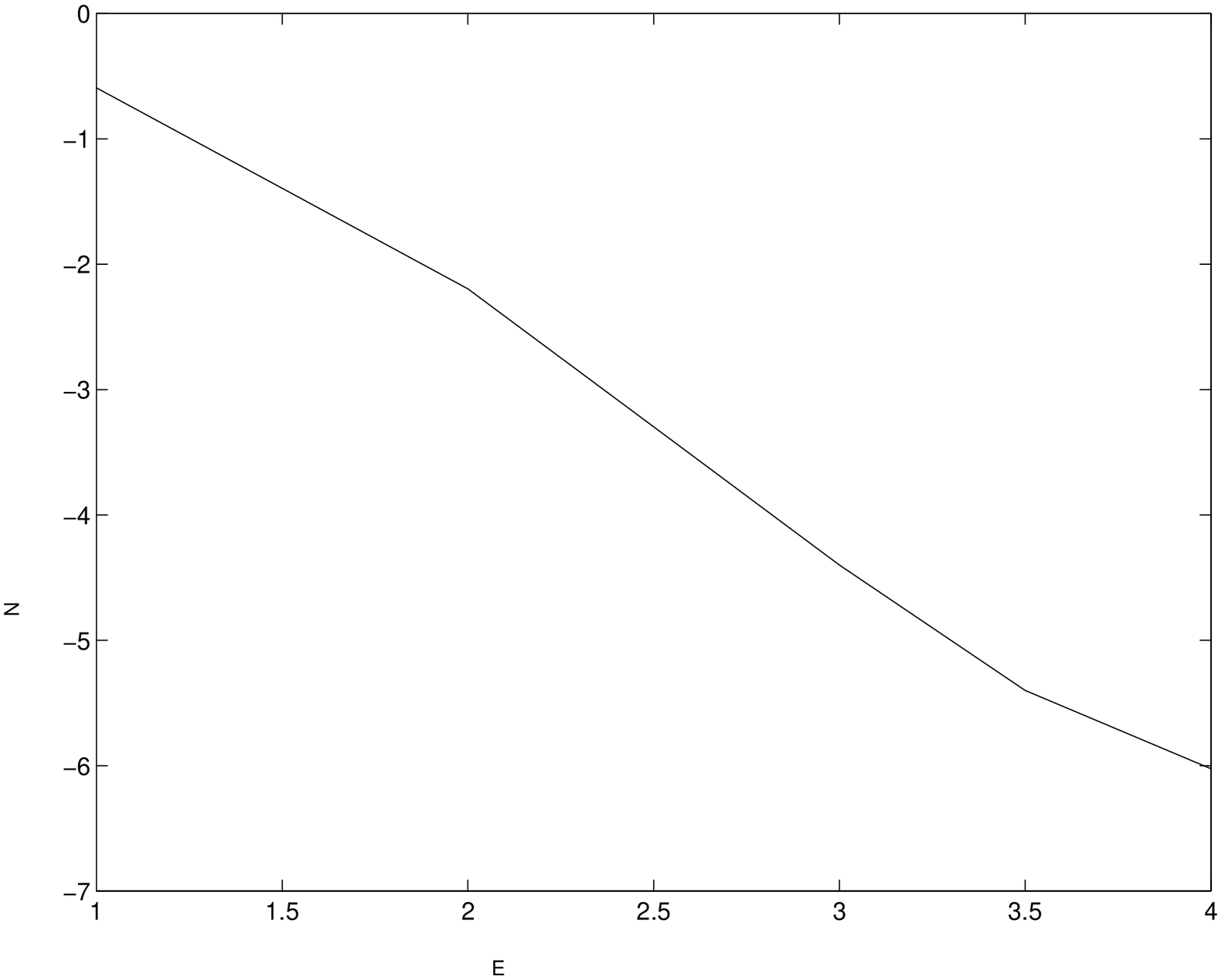}
\caption{(a) $\log_{10}\eta$ vs $-\log_{10}\eps$ for the various values of $\iota$ specified in Table \ref{T_varyIotaSlopes}. (b) Zoom in on $-\log_{10}\eps\in[1,4]$ for the value $\iota=162190$. The least squares approximation of the slope in the steepest part ($-\log_{10}\eps\in[2,3.5]$) is $-2.14$, on the whole interval $[1,4]$ it is $-1.89$.}\label{F_varyIota}
\end{center}
\end{figure}

\begin{table}[h]
\begin{center}
\begin{tabular}{c|ccccccc}
$\iota$ & 1200 & 4662 & 18236 & 40805 & 72239 & 112736 & 162190 \\ \hline
$Slope$ & -1.40 & -1.58 & -1.70 & -1.76 & -1.82 & -1.86 & -1.89 \\ \hline
$\max -\log_{10}\eps$ & 4 & 4.5 & 5 & 5& 6 & 6 & 6 
\end{tabular}
\caption{The second row is the least squares approximations of the slopes of $\log_{10}\eta(-\log_{10}\eps)$ on the domain $-\log_{10}\eps\in[1,4]$, for some different values of $\iota$. The third row gives the maximum value of $-\log_{10}\eps$, where the flow is slow, i.e., $s(\eps)\leq \frac{1}{\sqrt{\eps}}$.}
\label{T_varyIotaSlopes}
\end{center}
\end{table}

\subsection{Larger domain}\label{SS_LargeDomain}
In this subsection, we redo the experiment above for a square domain. There are roughly the same number of triangles in the $y$ and $z$ directions, rather than having only a couple of faces in each $\{y=const\}$ slice as we had in Subsection \ref{SS_Iota}. The resulting $\eta$ vs $\eps$ graph is given as Figure \ref{F_bigDomain}. We see that the results correspond to the coarser meshes in Figure \ref{F_varyIota}(a), which is natural, since a larger domain would require a larger number of faces. This illustrates that the results in Subsection \ref{SS_Iota} do not depend on the specific thin slice in the $z$-direction that we chose to study. For the two discretization sizes in Figure \ref{F_bigDomain} we have $\sqrt(\eps)s(\eps)\leq 1$ for $\eps\geq 10^{-4}$ and $\eps\geq 10^{-5}$, respectively. We are only able to prove that the cone fields are invariant for $\eps\leq 10^{-2.09}$, which means that for $\eps>10^{-2.09}$ the normal hyperbolicity is too weak for the algorithm to work. Thus, for the finest mesh size, we prove that the computable slow manifold exists for $10^{-5}\leq\eps\leq10^{-2.09}$.

\begin{figure}[h]
\begin{center}
\psfrag{E}{$-\log_{10}\eps$}
\psfrag{N}{$\log_{10}\eta$}
\includegraphics[width=0.45 \textwidth]{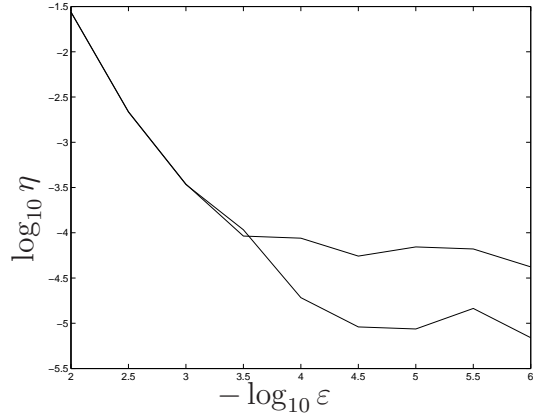}
\caption{$\log_{10}\eta$ vs $-\log_{10}\eps$ for the values $\iota=21810$ and $\iota=194396$. The least squares approximations of the slopes on the interval $[2,4]$ are $-1.27$ and $-1.52$, respectively.}\label{F_bigDomain}
\end{center}
\end{figure}

\subsection{The effect of the tightening step}\label{SS_Tight}
The tightening step is the slow part of the algorithm, and our program spends the vast majority of its computing time performing this step. It is therefore interesting to see how the results of a fast version of the algorithm, without the tightening step compares, performance wise. We run the example from Section \ref{SS_Iota}, with the highest precision ($\iota=162190$), and compare the results. The $\eta$ vs $\eps$ graph of the results is given as Figure \ref{F_with_OTightening}. In this example, the program spends $92.7\%$ of the computing time performing the tightening step. The total computing time in this case was $1526$ seconds on a 3.2 GHz Dual-Core AMD Opteron. 
For the example at hand, it might not be worth the extra effort to compute the tightening step or all applications. We do need it, however, for the application in Section \ref{S_Tang}.
 
\begin{figure}[h]
\begin{center}
\psfrag{E}{$-\log_{10}\eps$}
\psfrag{N}{$\log_{10}\eta$}
\includegraphics[width=0.45 \textwidth]{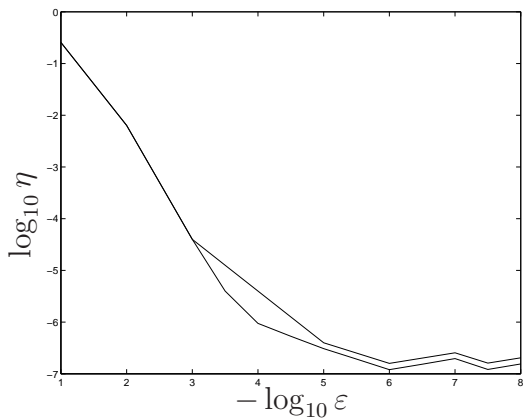}
\caption{$\log_{10}\eta$ vs $-\log_{10}\eps$ for the value $\iota=162190$, with and without the tightening step of the algorithm.}\label{F_with_OTightening}
\end{center}
\end{figure}

\subsection{Varying the improvement rate}\label{SS_ImpRate}
Our method contains a choice of the heuristic constant in the denominator of equation (\ref{eq_update}) that regulates the aggressiveness of the tightening step. In this subsection, we present a study on how the results depend on this choice. We use the same model as above, the domain from Subsection \ref{SS_Iota}, and the finest mesh size from Subsection \ref{SS_Iota} - $\iota = 162190$. For the purpose of this study, we denote the denominator of equation (\ref{eq_update}), by $l$. In Figure \ref{F_varyQuota} we display the results for $l=4,6,8$. For larger values of $l$, the results are virtually indistinguishable from the $l=8$ case. Typically, the updates mostly occur for smaller values of $\eps$. The reason is that for sufficiently small values of $\eps$ the vector field is almost equal to the layer equation, which makes the transversality condition almost trivial. Therefore, less smooth triangulations will still work, and the updates will not violate the transversality conditions.

\begin{figure}[h]
\begin{center}
\psfrag{E}{$-\log_{10}\eps$}
\psfrag{N}{$\log_{10}\eta$}
\includegraphics[width=0.45 \textwidth]{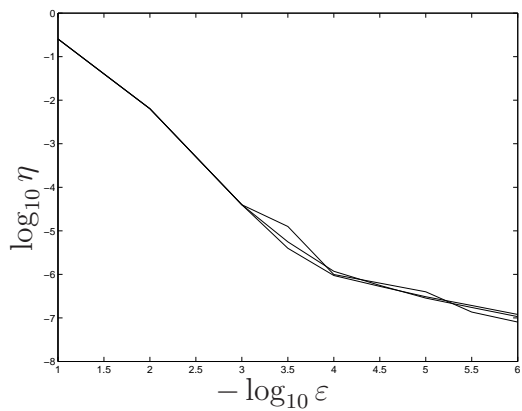}
\caption{$\log_{10}\eta$ vs $-\log_{10}\eps$ for the value $\iota=162190$, for updates with  $\|(v_R-v_L)\|$ divided by $4$, $6$, and $8$.}\label{F_varyQuota}
\end{center}
\end{figure}


\section{Tangencies}\label{S_Tang}
In this section we give a proof that the singular Hopf normal form, given by \eqref{eq_resc_shnf}, used here with $(A,B,C)=(-0.07,0.001,0.16)$, undergoes a tangency bifurcation of the unstable manifold of the saddle equilibrium, and the repelling slow manifold. We will often refer to these manifolds as the unstable manifold and the slow manifold denoted by $W^u_\mu$ and $S_\mu^r$, respectively. With a slow manifold for the rescaled system, we mean the image of a computable slow manifold for some $\eps$ under the map \eqref{eq_rescaling}.

Recall that (computable) slow manifolds are not unique. We therefore need to define what we mean by tangency, since if one choice of computable slow manifold is tangential, there will be other choices where the intersection is transversal. The natural setting is therefore to define when a one parameter family of slow manifolds is tangential to another manifold or family of manifolds.

\begin{definition}
A smooth one parameter family of manifolds, $\{M_\mu\}_{\mu\in[\mu_0,\mu_1]}$, intersects a one parameter family of families of computable slow manifolds $\{C_\mu\}_{\mu\in[\mu_0,\mu_1]}$ tangentially if for each choice of a smooth one parameter family of computable slow manifolds $\{S_\mu\}_{\mu\in[\mu_0,\mu_1]}$, $S_\mu\in C_\mu$, there is a value of $\mu\in(\mu_0,\mu_1)$ such that $S_\mu$ and $M_\mu$ intersect tangentially.
\end{definition}

In our proof, we compute one enclosing region $\C$ that satisfies the requirements from Section \ref{S_Existence} for all values of the parameter $\mu$ that appear in the proof. However, the computable slow manifolds might change with the parameter, since they are defined using \eqref{eq_hatSeps}. We prove that the one parameter family of unstable manifolds $W^u_\mu(p_\mu)$ moves through this fixed enclosing region, and that, as the family passes through $\C$, it always has to have a tangential intersection with at least one of the computable slow manifolds, regardless of how the smooth one parameter family of computable slow manifolds inside of $\C$ was chosen.   


\begin{theorem}\label{T_tangential}
For $0<\eps\leq 10^{-3}$, the singular Hopf normal form \eqref{eq_singularHopfNormalForm} undergoes a tangential bifurcation of a computable slow manifold and the unstable manifold of the equilibrium. The bifurcation occurs in the interval $[\mu_0,\mu_1]=[0.00454,0.004553]$ with fixed parameters $(A,B,C)=(-0.07, 0.001,0.16)$. 
\end{theorem}


The main argument in the proof of Theorem  \ref{T_tangential} is illustrated in Figure \ref{F_finalBox}. We consider the intersections of $S_\mu^r$ and $W^u_\mu$ with a half-plane $\Sigma$. At $\mu_0$, the two manifolds do not intersect each other in $\Sigma$. Notice that the unstable manifold seems to translate to the left relative to the repelling slow manifold  as $\mu$ increases. At $\mu_1$, the two manifolds intersect transversally in $\Sigma$. In the proof of the theorem, we formalize and prove these observations, and moreover show that the first intersection of the two manifolds is tangential. The vector field is transverse to $\Sigma$, so a tangential intersection of the manifolds in $\Sigma$ corresponds to a tangential intersection in the 3-dimensional phase space.

In the proof of Theorem \ref{T_tangential}, we will at times work with the singular Hopf normal form \eqref{eq_singularHopfNormalForm}, and at other times with the rescaled singular Hopf normal form \eqref{eq_resc_shnf}. Recall that, in the rescaled system we use upper case variables and parameters ($\mu$ is scale independent). Note that we do not assert that the tangency of the manifolds is unique. We will first prove Theorem \ref{T_tangential} for $\eps=10^{-3}$. For smaller $\eps$, the result follows from the rescaling \eqref{eq_rescaling} and the following property of the relative slope condition for the singular Hopf normal form (recall that the tangency will occur in different parts of phase space for different values of $\eps$):
$$
\sqrt{\eps}s(\eps) = \sqrt{\eps}\dfrac{2|X'|\sqrt{\eps Y}}{\sqrt{\eps}|Y'|+|Z'|} =  2|X'|\sqrt{Y}\dfrac{\eps}{\sqrt{\eps}|Y'|+|Z'|},
$$
since $\eps>0$, $|Y'| \geq 0$, and $|Z'|\geq 0$, this is a non-decreasing function of $\eps$.

Hence, if 
$$
s(\eps)\leq\dfrac{1}{\sqrt{\eps}},
$$
then
$$
s(\eps')=\dfrac{1}{\sqrt{\eps'}}\sqrt{\eps'}s(\eps')\leq \dfrac{1}{\sqrt{\eps'}}\sqrt{\eps}s(\eps) \leq \dfrac{1}{\sqrt{\eps'}}, \quad {\rm for \, all }\,\, 0<\eps'<\eps.
$$
The existence of computable slow manifold at a particular value of $\eps$ thus implies the existence of computable slow manifolds at all smaller values of $\eps$. Note that these computable slow manifolds will appear at different positions in the phase space for different values of $\eps$.

\begin{figure}[h]
\begin{center}
\includegraphics[width=0.7 \textwidth]{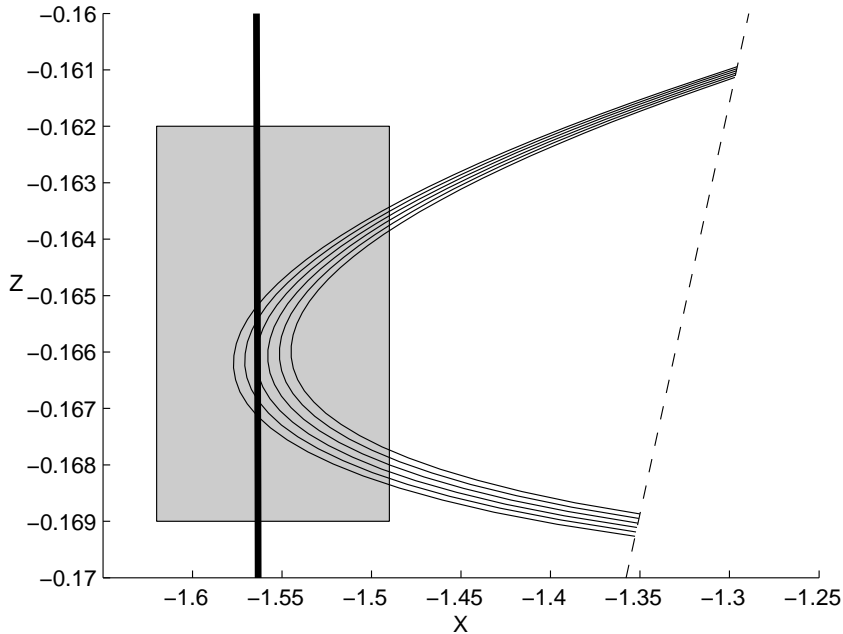}
\caption{An example family $S_\mu^r$ (thick line) and images of fundamental domains (solid curves) of $W^u_\mu$ for a selection of $\mu$ in $[\mu_0, \mu_1]$, shown here intersected with $\Sigma$. The boundary of $\Sigma$ is drawn as a dashed line, the rectangle $R$ is drawn as a shaded region. As $\mu$ is varied in $[\mu_0,\mu_1]$, the slow manifold only moves by amounts too small to be noticeable at the scale of the diagrams. 
For each $W^u_\mu$ included in the figure, we plot the first intersection of the trajectories with $\Sigma$, if the trajectory reaches $\Sigma$.}
\label{F_finalBox}
\end{center}
\end{figure}

{\bf Set-up. } 
We will work with 
$$
\Sigma :=\{(X,Y,Z)\in \RR^3: Z \geq -0.1693+0.16\, (X+1.353), Y=2\}
$$ 
and
$$
R:=\{(X,Y,Z)\in\Sigma:-1.62\leq X \leq -1.49, -0.169\leq Z \leq -0.162\}.
$$
We next list verifiable conditions that via Lemma \ref{L_tangential} below will prove Theorem \ref{T_tangential}. Many of these conditions are illustrated in Figure ~\ref{F_initialFinalBox}. Let 
$$
Y_{min},Y_{max}: [ \mu_0,\mu_1]\rightarrow \mathbb{R}
$$
be continuous with $Y_{min}(\mu)\leq Y_{max}(\mu)$. Further define a 2-dimensional ``box'' by
$$
B_0:=\{(\mu,X,Y,Z) \in  [\mu_0,\mu_1]\times W^u_\mu:\, X= \pi_X(p_\mu), Y_{min}(\mu)\leq Y\leq Y_{max}(\mu)\}.
$$ 
Note that the requirement $(X,Y,Z)\in W^u_\mu$ uniquely defines $Z$ as a function of $(\mu,X,Y)$. Denote the corners of $B_0$ corresponding to 
$$(\mu,Y)\in\{(\mu_1,Y_{max}(\mu_1)), 
(\mu_0,Y_{max}(\mu_0)),
(\mu_0,Y_{min}(\mu_0)), 
(\mu_1,Y_{min}(\mu_1))\}$$
by $\{M_1, M_2, M_3, M_4\}$. Denote the flow map of system \eqref{eq_resc_shnf} from $B_0$ to $\Sigma$, wherever it is defined, by $\Psi$. The next step of our construction is to introduce a number of assumptions, that are verifiable using validated numerics, i.e., they can be restated as a finite number of computable conditions. The geometry of these assumptions is illustrated in Figure \ref{F_initialFinalBox}. In Lemma \ref{L_tangential} below we show that these assumptions are sufficient to prove Theorem \ref{T_tangential}.

\begin{assumption}\label{A_Tangency}
Assume that the following conditions are satisfied:
\begin{enumerate}
\item[(I)] For $\mu\in[\mu_0,\mu_1]$, a family of repelling slow manifolds $S_\mu^r$ intersects $R$ in a single family of curves $C_\mu$ that enters $R$ at the top and exits $R$ at the bottom.
 
\item[(II)] The map $\Psi$ is defined on the three sides of $B_0$  corresponding to $Y=Y_{min}(\mu), Y=Y_{max}(\mu)$ and $\mu=\mu_0$, and their images under $\Psi$ lie in $R$ and strictly to the right of $S_\mu^r\cap R$.

\item[(III)]  The map $\Psi$ is defined on $\{ (\mu, X,Y,Z)\in B_0 : \mu=\mu_1\}$ and its image lies in $R$. Furthermore, $\Psi(M_1)$ and $\Psi(M_4)$ are strictly to the right of $S_\mu^r\cap R$, and there exists a point $M_5\in\{ (\mu, X,Y,Z)\in B_0 : \mu=\mu_1\}$ such that 
$\Psi(M_5)$ lies strictly to the left of  $S_\mu^r\cap R$ in $\Sigma$.

\item[(IV)] The map $\Psi$ is well-defined on $B_0$.
\end{enumerate}
\end{assumption}

\begin{figure}[h]
\begin{center}
(a)\includegraphics[width=0.46 \textwidth]{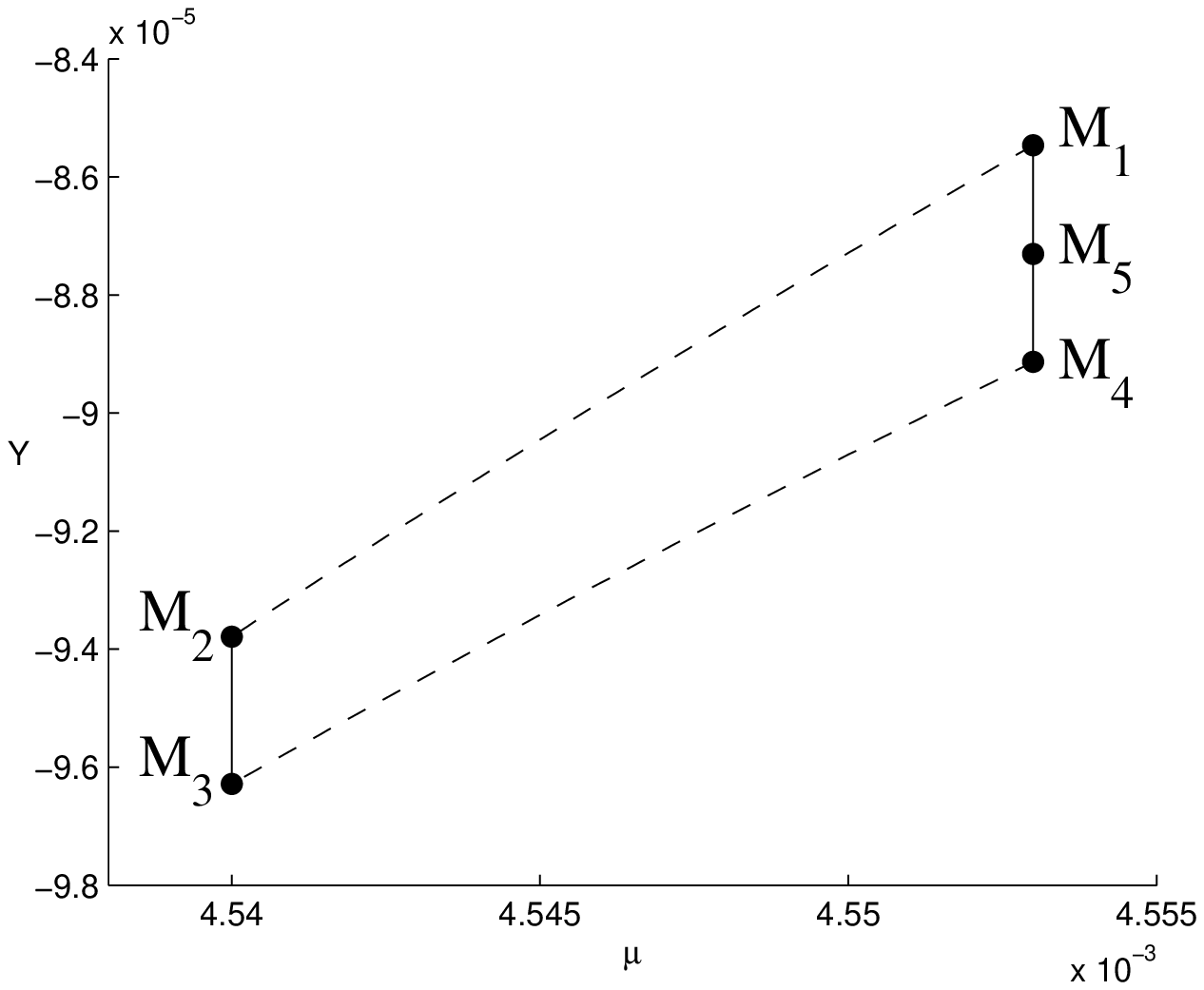}
(b)\includegraphics[width=0.46 \textwidth]{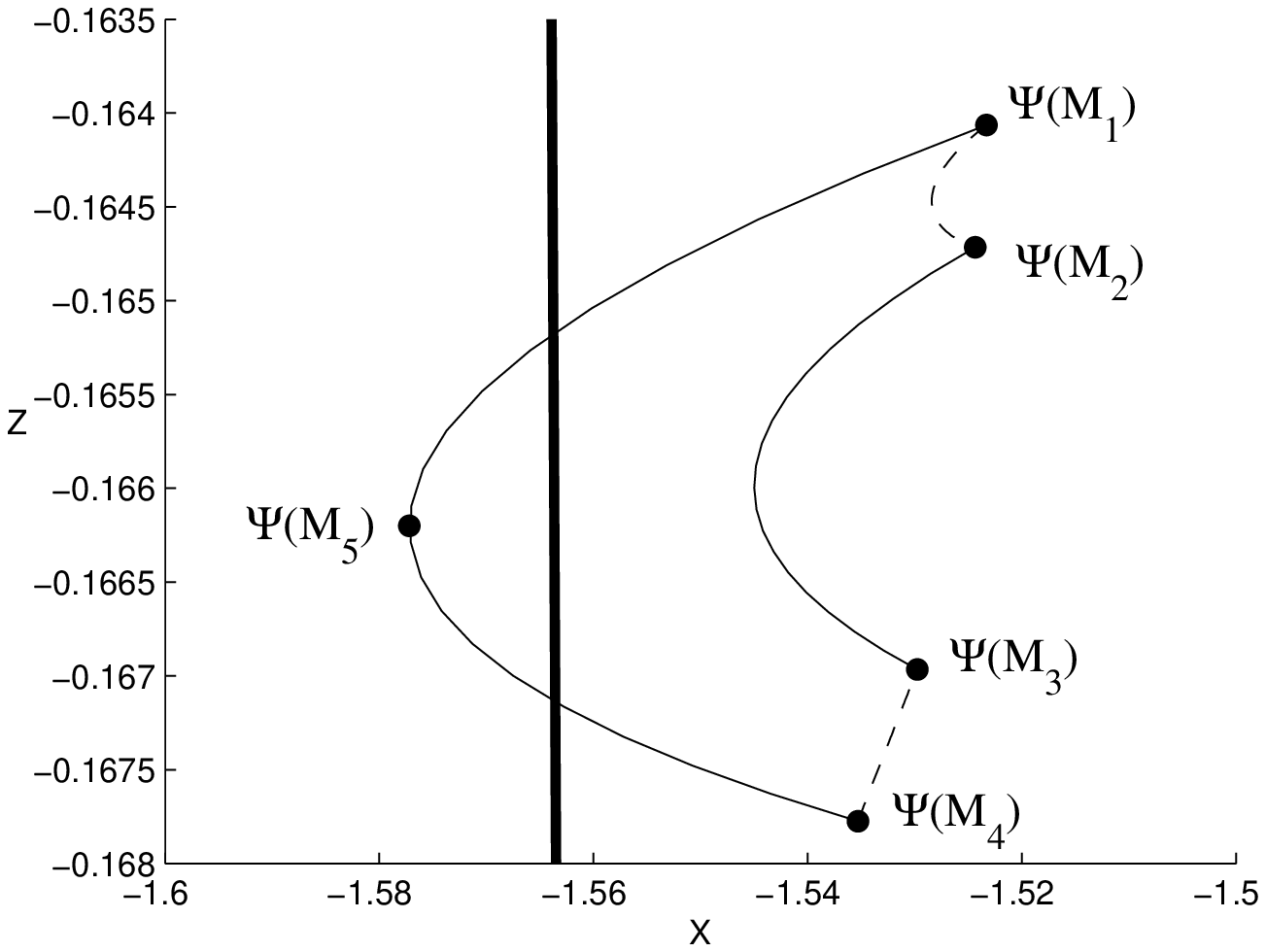}
\caption{Illustration of the assumptions made in Assumption \ref{A_Tangency}. The box $B_0$ shown in pane (a) maps into $R\subset\Sigma$ as shown in pane (b). As $\mu$ is varied in $[\mu_0,\mu_1]$, the slow manifold (thick solid line) only moves by amounts too small to be noticeable at the scale of the diagrams.}\label{F_initialFinalBox}
\end{center}
\end{figure}

\begin{lemma} \label{L_tangential} Suppose that Assumptions \ref{A_Tangency} are satisfied. Then $S^r_{\mu}$ and $W^u_\mu$ intersect tangentially for some $\mu^* \in [\mu_0,\mu_1]$.
\end{lemma}

\begin{proof}[Proof of Lemma \ref{L_tangential}] 
Fix a family of slow repelling manifolds $S^r_{\mu}$, $\mu\in[ \mu_0, \mu_1]$. Since all of $B_0$ reaches $\Sigma$ by Assumption \ref{A_Tangency}.IV, the existence and uniqueness theorem for ODEs implies that the map from $B_0$ to $\Sigma\times [ \mu_0, \mu_1]$ is continuous. We may thus define the continuous function 
\begin{equation*}
\mathrm{dist}(\mu,Z):=\min \left(\pi_X (W^u_\mu|_Z\cap R) - \pi_X (S^r_{\mu} |_Z\cap R) \right)
\end{equation*}
where  $Z$ is required to lie in the range of $Z$ values of  $R$ and $|_Z$ denotes restriction to $Z$. Consider 
$$\mu^* = \min \{\mu\in [\mu_0,\mu_1]:\, \min_Z \mathrm{dist}(\mu,Z)=0\},
$$
the existence of which follows from Assumptions \ref{A_Tangency}.II and \ref{A_Tangency}.III, and the continuity of $\mathrm{dist}(\mu,Z)$. Clearly $S_{\mu^*}^r\cap R$ and $W^u_{\mu^*}\cap R$ intersect in at least one point $(X_0,Y_0,Z_0)$. Moreover, 
$$
\pi_X(W^u_{\mu^*}|_Z\cap R) - \pi_X(S_{\mu^*} |_Z\cap R)\geq 0
$$ 
for the range of $Z$ values that lie in $R$. Since $W^u_{\mu^*}$ and $S_{\mu^*}$ are smooth surfaces in $\RR^3$, transverse to $R$, we can now consider the Taylor series expansion of 
$$
\pi_X(W^u_{\mu^*}|_Z\cap R) - \pi_X(S_{\mu^*} |_Z\cap R)
$$
 at $X_0$ with respect to $X$ and conclude that its linear term must be zero. We have thus shown that the manifolds $W^u_{\mu^*}$ and $S_{\mu^*} $  intersect tangentially in $R$.
\end{proof}

Subsection \ref{SS_SlowManifold} below gives details on the verification of Assumption \ref{A_Tangency}.I. Subsection \ref{SS_UnstableManifold} describes in detail how $\Sigma$, $Y_{min}$, and $Y_{max}$ are chosen, and provides details on the verification of Assumptions \ref{A_Tangency}.II, \ref{A_Tangency}.III, and \ref{A_Tangency}.IV. 


\subsection{Slow manifold computations.} \label{SS_SlowManifold}
Showing that for $\mu\in [\mu_0,\mu_1]$, a family of repelling slow manifolds $S_\mu^r$ intersects $R$ in a single family of curves is a straight-forward application of the methods developed in the earlier parts of this paper: we compute slow manifold enclosures for the rescaled singular Hopf system \eqref{eq_resc_shnf} over a domain that corresponds to 
$$
1\leq Y \leq 500, \quad -0.169\leq Z \leq -0.162,
$$
and for the singular Hopf parameter values
$$
\{(\mu, A, B, C)\in \mathbb{R}^4: \mu\in [4.54,4.553]\times 10^{-3}, A=-0.07, B=0.001, C=0.16\}.
$$
The actual computations for the enclosures are performed in the original singular Hopf coordinates of \eqref{eq_singularHopfNormalForm}, as described in earlier sections of this paper. Let $\epsilon_0 = 10^{-3}$, in the original coordinates of the singular Hopf normal form, \eqref{eq_singularHopfNormalForm}, the domain now corresponds to 
$$
D=[y_{min}, y_{max}]\times [z_{min}, z_{max}],
$$ 
where
$y_{min}=1.0\, \epsilon_0,y_{max}=500.0 \, \epsilon_0,z_{min}=-0.169\, \sqrt{\epsilon_0},z_{max}=-0.162\, \sqrt{\epsilon_0}$, and the set of singular Hopf system parameters is
$$
\left\{(\epsilon,\mu,a,b, c)\in \mathbb{R}^5:  \epsilon=\epsilon_0, \mu\in [4.54,4.553]\times 10^{-3}, a=-\frac{0.07}{\sqrt{\epsilon_0}}, b=\frac{0.001}{\epsilon_0}, c=\frac{0.16}{\sqrt{\epsilon_0}} \right\}.
$$
The enclosures obtained show that points $(X,Y,Z)$ in the repelling slow manifold over $D$ must satisfy 
$$
-1.5726<X<-1.5539.
$$
Moreover, the methods of Section \ref{S_Existence} of this paper were used to check that at any parameter in the above-described set,  $S_\mu^r$ is a graph over a domain $D \subset S_0$, and that $s(\eps_0)\leq1.027$. Again, note that the computation is independent of the choice of $\eps_0$, since a different choice of $\eps_0$ would imply that we should enclose a different part of the phase space. Since the above $z_{min}$ and $z_{max}$ were chosen large and small enough, respectively, to conclude that the enclosed repelling slow manifolds enter $R$ at the top and leave $R$ at the bottom, we have shown the existence of the sought family of slow manifolds.

\noindent {\bf Remark.} Even though the slow manifold intersected with the section $\Sigma$ in our case resembles a fixed straight line, enclosing it with the precision required for the proof to work is a hard problem. To determine rigorously the location of a slow manifold is difficult even in the easiest non-trivial cases. The problem is amplified in our case since we need high accuracy in the rescaled system, where the errors are blown up by a factor $O\left(\frac{1}{\sqrt{\eps}}\right)$.

\subsection{Unstable manifold computations.}\label{SS_UnstableManifold} 
We now describe how $Y_{min}$ and $Y_{max}$ are chosen for Assumptions \ref{A_Tangency}.II and \ref{A_Tangency}.III to be satisfied. Recall that Figure \ref{F_finalBox} was obtained by examining trajectories in entire fundamental domains of $W^u_\mu$ for $\mu\in[\mu_0, \mu_1]$, and that 
some of these trajectories did not reach $\Sigma$. We chose $Y_{min}$ first and then $Y_{max}$ in such a way that $B_0$ is on the one hand small enough for the map to $\Sigma$ to be well-defined and its image to be in $R$, and on the other hand large enough for the images of marked points $M_1, \ldots, M_5$ and of the boundaries of $B_0$ to map to the left or right of $S_\mu^r$ as required by Assumptions \ref{A_Tangency}.II and \ref{A_Tangency}.III.

\subsubsection{Computing $Y_{min}(\mu)$.} 
Let $L\subset \Sigma$ be the line given by 
$$
L:=\{(X,Y,Z)=(t, 3, -0.1678+0.16\,(t+1.535))\,: t\in\RR\}.
$$
This line lies well within $\Sigma$ and is parallel to $\partial \Sigma$. It is moreover transverse to the parts $W^u_\mu$ that reach $\Sigma$, for all $\mu\in[\mu_0, \mu_1]$. The boundary value problem (BVP) for the flow of the rescaled singular Hopf normal form with the following boundary conditions and $\mu$ as the continuation parameter is thus well-defined: 

\begin{itemize}
\item trajectories have to start in the unstable eigenspace of $p_\mu$
\item trajectories have to start at an $X$ coordinate equal to that of $p_\mu$
\item trajectories have to end on $L$. 
\end{itemize}
Note that there are multiple solutions to this BVP, as each trajectory in the unstable manifold satisfies the initial boundary condition multiple times as it spirals away from the equilibrium point $p_\mu$, but we choose one by selecting a fundamental domain for its endpoint near $p_\mu$.  The equilibrium points $p_\mu=(x_\mu,x_\mu^2,x_\mu)$ satisfy the equation $x_\mu=-45+\sqrt{45^2-1000\mu}$.
We use a trajectory that initially has a $Y$ coordinate approximately $10^{-4}$
larger than that of $p_\mu$, deferring a discussion of the suitability of this distance to a remark at the end of this subsection. Solving the BVP with a shooting method, we find that the $Y$ coordinates of the solutions to the boundary value problem are close to linear in $\mu$ on the interval $[\mu_0, \mu_1]$. We thus define 
$$
Y_{min}(\mu)=x_\mu^2-9.37888799540\times 10^{-5} + 0.640307054861539(\mu-\mu_0),
$$
to be the linear function in  $\mu$ that approximates the $Y$ coordinates of the solution endpoints to the BVP.

\subsubsection{Computing $Y_{max}(\mu)$.} 
After inspecting diagrams similar to Figure \ref{F_initialFinalBox}, we defined $Y_{max}(\mu)$ in an ad-hoc manner to be the linear function for which the box $B_0$ contains $25\%$ of a fundamental domain of $W^u_{\mu_0}$ and $40\%$ of a fundamental domain of $W^u_{\mu_1}$:
$$
Y_{max}(\mu)=x_\mu^2 -9.628167607168\times 10^{-5} + 0.549805711513847(\mu-\mu_0).
$$

\subsubsection{Computing unstable manifolds.}\label{SSS_WU} 
The complexity of the singular Hopf normal form makes it unfeasible to compute unstable manifolds analytically. We therefore begin by describing a method to rigorously compute the location of $W^u_\mu$. We will use the method developed in Section \ref{S_Method} together with covering relations with cone conditions \cite{Z09} and validated numerical integration \cite{L88,NJ99,NJC99,NJP01} to enclose and propagate the manifolds, respectively. In our implementation \cite{Progs} we use the software VNODE-LP \cite{Vnode} to integrate the system \eqref{eq_resc_shnf}. The computations are done using order $11$ Taylor expansions in VNODE-LP.

Since our proof relies heavily on the concept of h-sets and the method of covering relations we provide an informal introduction here. For a complete formal description of these concepts and methods we refer the reader to \cite{ZG04,Z09}. In \cite{ZG04} h-sets and covering relations are introduced, and in \cite{Z09} the concept of an h-set with cones is introduced together with the appropriate modification to the definition of a covering relation. An h-set is a compact hyperbolic like set, in the sense that it has expanding and contracting directions, in an appropriate coordinate system. An h-set is a set together with the coordinates. A map together with two h-sets, $h_1, h_2$, is said to satisfy covering relations if $h_1$ is mapped across $h_2$ under the map. Across in this setting means that the boundaries of $h_1$ transversal to the expanding directions are mapped outside $h_2$ and the image of $h_1$ does not intersect the boundaries of $h_2$ transversal to the contracting directions. Using the Brouwer degree one can show, see \cite{ZG04}, that a cycle of h-sets with covering relations must contain a periodic orbit. An h-set with cones is an h-set together with a quadratic form $Q$, that describes a uniform cone field on the h-set. The map is said to satisfy covering relations with cone conditions, if the quadratic form is increasing along orbits. Given recurrence, this yields uniqueness of periodic orbits. One can also use the cone conditions, see \cite{Z09}, to prove the existence of invariant manifolds and propagate them along orbits, which is how they are used in this section. The bounds on the location of the invariant manifolds given by covering relations with cone conditions are Lipschitz. In particular around a fixed point one gets a cone, which bounds the location of the invariant manifold. The Lipschitz constant depends on the ratio of the positive and negative eigenvalues of $Q$.

We construct an $h$-set with cones 
centered at $p_\mu$ as a cylinder of size $10^{-4}$ and $10^{-5}$ in the $(X,Y)$ and $Z$ directions, respectively, with a cone with Lipschitz constant $0.1$ defined by the quadratic form
$$
Q=\left[\begin{array}{ccc}
1 &0&0\\0&1&0\\0&0&-100
\end{array}
\right].
$$ 
We verify that covering relations and cone conditions hold for the time $6.3$ map. This proves that the unstable manifold exists within the $h$-set, and yields an enclosure of the unstable manifold as a Lipschitz graph with Lipschitz constant $0.1$ over the disc:
$$
\left\{(X,Y) \,: \|(X-x_\mu,Y-x_\mu^2)\|\leq 10^{-4}\right\}.
$$ 
To further contract the enclosure for a given value of $(X,Y)$, we partition the line segment over $(X,Y)$ in the cone, and integrate backwards for $100$ time units or until the trajectory passes the cone. Subsegments that leave the cone in backwards time are removed, and we use the interval hull of the remaining subsegments as our new bound of a point in the unstable manifold. The covering relations with cone conditions prove that each remaining subsegment over $(X,Y)$ contains a unique $Z$ value such that $(X,Y,Z)\in W^u_\mu$.

Given an initial enclosure of a point in $W_\mu^u$ we propagate it forwards by integrating \eqref{eq_resc_shnf} until it hits $\Sigma$ using VNODE-LP. To integrate the top and bottom of $B_0$, i.e., the boundaries of $B_0$ where $\mu$ is not constant, and the interior of $B_0$, we consider a $4$ dimensional phase space by appending $\dot \mu = 0$ to \eqref{eq_resc_shnf}. This procedure stabilizes the numerical behavior of the propagation of the unstable manifold.

\subsubsection{Verifying Assumptions \ref{A_Tangency}.(II-IV)} 
Using the method described in Section \ref{SSS_WU} one can now subdivide $\partial B_0$ into small subsets, compute an interval enclosure of each subset, and use validated numerical integration to show that Assumptions \ref{A_Tangency}.II, \ref{A_Tangency}.III, and \ref{A_Tangency}.IV are satisfied. In practice, this requires some experimentation: if the subsets are too large, wrapping effects in the numerical integration will make the verification of Assumptions \ref{A_Tangency}.II, \ref{A_Tangency}.III, and \ref{A_Tangency}.IV impossible. On the other hand, the computing time for the entire verification of Assumption \ref{A_Tangency}.II is approximately proportional to the  number of subsets to be integrated numerically. The bounds on $\Psi(\partial B_0)$ and $\Psi(M_i)$, for $i=1,4,$ and $5$, are given in Table \ref{T_PsiIm}.

\begin{table}[h]
\begin{center}
(a) \begin{tabular}{c|cccc}
& $\Psi(\partial B_0(\mu_0))$ & $\Psi(\partial B_0(\mu_1))$ & $\Psi(\partial B_0(Y_{min}))$ & $\Psi(\partial B_0(Y_{max}))$ \\ \hline
$X$ & $-1.5_{227}^{468}$ & $-1.^{6102}_{5156}$ & $-1.5_{236}^{462}$ & $-1.5_{107}^{368}$\\ \hline
$Z$ & $-0.16_{46}^{71}$ & $-0.16_{35}^{83}$ & $-0.16_{64}^{84} $ & $-0.16_{29}^{57}$
\end{tabular}

\vspace{0.3cm}

(b)\begin{tabular}{c|ccccc}
& $\psi(M_1)$ & $\psi(M_4)$ & $\psi(M_5)$\\ \hline
$X$ & $-1.52_{29}^{37}$ & $-1.535_{1}^{5}$ & $-1.57_{58}^{86}$ \\ \hline
$Z$ & $-0.164_0^1$ & $-0.167_7^8$ & $-0.166_1^3$
\end{tabular}

\vspace{0.3cm}

\caption{(a) The image of $\partial B_0$ under $\Psi$. (b) The image of the marked points on the $\partial B_0(\mu_1)$ line under $\Psi$. All images are in the interior of $R$. The computations in (a) and (b) prove Assumptions \ref{A_Tangency}.II and \ref{A_Tangency}.III, respectively.}\label{T_PsiIm}
\end{center}
\end{table}


{\bf Remark.}
Note that since the position of $S_\mu^r$ as well as the map to $\Sigma$ are computed using interval arithmetic, their computed positions have errors due to over estimation associated with them. These errors have to be taken into account when choosing the $Y$ value at which to place the half-plane $\Sigma$, the interval boundaries $\mu_0$ and $\mu_1$, and the functions $Y_{min}$ and $Y_{max}$. Generally,  placing $\Sigma$ at greater values of $Y$ results in tighter bounds for the slow manifold, and the repelling nature of the slow manifold spreads trajectories that were initially close in the fundamental domain far apart, making it easier to verify Assumptions \ref{A_Tangency}.(II-IV). We found the size $2\times 10^{-4}$ of the $h$-sets constructed in Section \ref{SSS_WU} to be large enough to keep the validated numerical integration to $\Sigma$ short enough to not accumulate prohibitively large errors, while being small enough to be efficiently computable.

{\bf Remark.} To give further insight into what happens after the bifurcation we note that the following set is forward invariant. For other values of the parameters, similar sets can be constructed. For $k>B$, 
$$
X<\frac{-\mu}{A+C}, \quad X^2>(1+k)Y, \quad Y>\frac{1+k}{k}, \quad |X|>|Z|. 
$$

We verify that the above conditions are satisfied, with $k=2$, for the point $M_5$. Thus, $X\rightarrow-\infty$ and $Y\rightarrow\infty$ for a part of the unstable manifold past the tangential bifurcation.




\section{Summary and Discussion}\label{S_Disc}

Computation of the slow manifolds in a normal form for singular Hopf bifurcation served as a case study for this paper. A singular Hopf bifurcation in slow-fast systems with two slow and one fast variable occurs when an equilibrium point crosses between attracting and repelling slow manifolds. The dynamics associated with this crossing -- a \emph{folded saddle-node type II} in the singular limit -- is complicated. The small amplitude oscillations emanating from the equilibrium point are part of \emph{mixed mode oscillations} in some examples, notably the model originally studied by Koper. Subsidiary bifurcations occur, including tangency between the repelling slow manifold and the two dimensional unstable manifold of the equilibrium point. Tangency bifurcations form part of the boundary of the parameter space region in which mixed mode oscillations occur in the Koper model, making them essential to understanding global aspects of the dynamics in this and other systems. Since there are no analytic methods for locating the tangency bifurcations, this paper uses verified computing methods to prove the existence of tangency bifurcations between a slow manifold and an unstable manifold of an equilibrium point for the first time.


Some of our ideas generalize to the case of slow manifolds of saddle type. To compute normally hyperbolic manifolds of saddle type, see e.g. \cite{CZ11}, one usually first computes the manifold's stable and unstable manifolds, and then intersects them. To compute a saddle slow manifold in a three dimensional ambient space using our ideas, one could compute enclosures of the stable and unstable manifolds, as presented in this paper. The existence argument given in Section \ref{S_Existence} can be modified to this setting, under appropriate assumptions on the dynamics on the slow manifold. Generalization to slow manifolds of saddle type in higher dimensional ambient spaces is substantially more challenging.

We made several design decisions while constructing our algorithm for computing slow manifolds. This section discusses details of several and motivates our choices. 
\begin{itemize}
 \item 
Our enclosures were constructed as pairs of enclosing transversal piecewise linear surfaces. There are several alternative approaches to how to construct and refine the vertices of the enclosing triangulated surfaces $\L$ and $\R$. For the examples in Sections \ref{S_NumRes} and \ref{S_Tang} we used rectangular patches in the domain of the slow variables. Instead, one could construct the triangulations of the original domain in the slow variables by considering a dynamically defined region, constructed by flowing a set of initial conditions on the critical manifold with the slow flow, and use a discretization of those trajectories as the vertices of the triangulation.
\item
We considered other possibilities for moving vertices in Section \ref{SS_ImpBound}; namely, to move them along trajectories of the flow of (\ref{eq_slowFast_12}), or to move them along the normal of the triangulation. Both of these methods have serious disadvantages. When moving vertices along the flow of the system, we have to carefully check whether the vertices are moved past edges, thereby destroying the integrity of the triangulation. If the triangulation remains a graph over the $(y,z)$ domain, it is possible to generate a new triangulation by a Delaunay-type algorithm, and lift it to the surface, but if two vertices flow to the same $(y,z)$ coordinate this is no longer possible. Additionally, this method of moving vertices moves the two enclosing surfaces by different amounts, so that we obtain an enclosure of a smaller part of the slow manifold. A third drawback is that the triangulations might develop very acute triangles. Finally, numerical integration of a large number of vertices is slow compared to the approach that we use. Moving vertices along the normals combines the worst of both methods: we no longer control the triangulations, and we might introduce violations of the transversality conditions.    
\item
The tightening procedure described in Section \ref{SS_ImpBound} only updates one vertex at the time, i.e., we move one vertex a big step and if all the faces attached to it are still transversal to the flow, then we move it. An alternative would be to move not only the vertex itself, but at the same time all vertices attached to it by an edge. Such a procedure would work as follows: when it is one vertex' ``turn'', only update it by a fraction of its potential improvement, and simultaneously move the ones it attaches to, by a smaller amount. The smaller neighbour updates should be such that the expected value of the total update of each vertex stays the same as in Section \ref{SS_ImpBound}. The benefit of such an approach is that the triangulation is not skewed as much in each step, so it should be easier to verify the transversality condition. In practice, however, the gain of this approach is negligible, compared to a slight increase of the denominator of (\ref{eq_update}). There are also disadvantages of such an approach, primarily in its computational complexity. Each time an update is made, one has to not only locate all its neighbouring vertices and update them, but also locate all of their neighbouring faces and check the transversality condition on them. In the results presented in Section \ref{S_NumRes}, we thus only update one vertex at the time.
\item
We construct invariant cone fields on $\C$ to prove that it contains normally hyperbolic locally invariant manifolds. We constructed these manifolds by flowing a ``ribbon'' around the inflowing boundaries of the enclosure. The property that our enclosures were aligned with the flow in the sense that for one of the slow variables the vector field is non-zero, was crucial for proving the existence of computable slow manifolds. In general one could also use the invariant cone fields to show that the graph transform is well defined, by adapting the method in \cite{KH95}. To prove the convergence of such a scheme would require very careful estimates of the expansion and contraction rates, and the norms of the nonlinear components of the vector field. An alternative is to define an extension of the vector field outside of $\C$ that has a slow manifold that is invariant rather than just locally invariant. Global invariance together with normal hyperbolicity would give a unique manifold for the extension using the technique from \cite{CZ11}. Given normal hyperbolicity, ensured by the existence of the cone field, either method would give the existence of a (non unique) $C^1$ normally hyperbolic manifold, which is the graph over the slow variables. Either of these approaches, however, include many subtle details that need to be clarified for the case at hand. 
\item
If the mesh size of piecewise linear enclosing surfaces remains fixed as $\eps$ decreases, then the curvature of the slow manifold becomes a a limiting factor in the tightness of enclosures. With smoother enclosing manifolds, tighter enclosures are likely to be possible. We did not attempt this because the  transversality calculations for piecewise linear systems were particularly simple in the singular Hopf normal form we studied.
\end{itemize}

\section{Acknowledgment}
T. J. was funded by a postdoctoral fellowship from \textit{Vetenskapsr\aa det} (the Swedish Research Council). J. G. and P. M. were partially supported by a grant from the National Science Foundation.




\begin{thebibliography}{10}

\bibitem{AM06} 
{\sc Z.~Arai and K.~Mischaikow}, 
{\em Rigorous computations of homoclinic tangencies},
SIAM J. Appl. Dyn. Syst.,  5  (2006),  pp.~280--292.

\bibitem{CFL05}
{\sc X.~Cabr\'{e}, E.~Fontich, and R.~de la Llave}, 
{\em The parameterization method for invariant manifolds. III. Overview and applications},
J. Differential Equations,  218  (2005), pp.~444--515.

\bibitem{C09} 
{\sc M.~J.~Capi\'nski}, 
{\em Covering relations and the existence of topologically normally hyperbolic invariant sets},
Discrete Contin. Dyn. Syst.,  23  (2009), pp.~705--725. 

\bibitem{CZ11} 
{\sc M.~J.~Capi\'nski and P.~Zgliczy\'nski}, 
{\em Cone conditions and covering relations for topologically normally hyperbolic invariant manifolds},
Discrete Contin. Dyn. Syst., 30  (2011),  pp.~641--670. 

\bibitem{DGKKO}
 {\sc M.~Desroches, J.~Guckenheimer, B.~Krauskopf, C.~Kuehn, H.~M.~Osinga, and M.~Wechselberger},
{\em Mixed-mode oscillations with multiple time scales},
SIAM Review, to appear.

\bibitem{DKO08} 
{\sc M.~Desroches, B.~Krauskopf, and H.~M.~Osinga}, 
{\em Mixed-mode oscillations and slow manifolds in the self-coupled FitzHugh-Nagumo system},
Chaos, 18 (2008).

\bibitem{EKO07} 
{\sc J.~P.~England, B.~Krauskopf, and H.~M.~Osinga},  
{\em Computing two-dimensional global invariant manifolds in slow-fast systems},
Internat. J. Bifur. Chaos Appl. Sci. Engrg.,  17  (2007),  pp.~805--822. 

\bibitem{G95} 
{\sc J.~Guckenheimer}, 
{\em Phase portraits of planar vector fields: computer proofs},
Experiment. Math.,  4  (1995), pp.~153--165.

\bibitem{G08} 
{\sc J.~Guckenheimer}, 
{\em Singular Hopf bifurcation in systems with two slow variables},
SIAM J. Appl. Dyn. Syst.,  7  (2008),  pp.~1355--1377.

\bibitem{GK09} 
{\sc J.~Guckenheimer and C.~K\"uhn}, 
{\em Computing slow manifolds of saddle type},
SIAM J. Appl. Dyn. Syst.,  8  (2009),  pp.~854--879.

\bibitem{GM11} 
{\sc J.~Guckenheimer and P.~Meerkamp}, 
{\em Unfoldings of singular Hopf bifurcation},
preprint (2011), e-print: arXiv:1107.3185.

\bibitem{HPS77} 
{\sc M.~W.~Hirsch, C.~C.~Pugh, and M.~Shub}, 
{\em Invariant manifolds},
Lecture Notes in Mathematics, Vol. 583. Springer-Verlag, Berlin-New York, 1977.

\bibitem{HW} 
{\sc E.~Hairer and G.~Wanner}, 
{\em Solving ordinary differential equations. II. Stiff and differential-algebraic problems},
Second edition, Springer Series in Computational Mathematics, Springer-Verlag, Berlin, 1996, xvi+614~pp.

\bibitem{I00} 
{\sc E.~M.~Izhikevich}, 
{\em Neural excitability, spiking and bursting},
Internat. J. Bifur. Chaos Appl. Sci. Engrg., 10 (2000), pp.~1171--1266.

\bibitem{JT11a} 
{\sc T.~Johnson and W.~Tucker}, 
{\em A note on the convergence of parametrised non-resonant invariant manifolds},
Qualitative Theory of Dynamical Systems, 10 (2011), pp.~107--121.

\bibitem{JT11b} 
{\sc T.~Johnson and W.~Tucker}, 
{\em On a computer-aided approach to the computation of Abelian integrals}, 
BIT - Numerical Mathematics, 51 (2011), pp.~653--667. 

\bibitem{JT11c} 
{\sc T.~Johnson and W.~Tucker}, 
{\em On a fast and accurate method to enclose all zeros of an analytic function on a triangulated domain}, 
in Proceedings of PARA - 2008, to appear in Lecture Notes in Computer Science 6126/6127 Springer-Verlag, 2011.

\bibitem{J95} 
{\sc C.~K.~R.~T.~Jones}, 
{\em Geometric singular perturbation theory},
Dynamical systems (Montecatini Terme, 1994), Lecture Notes in Math., 1609, Springer, Berlin, 1995, pp.~44--118.

\bibitem{KH95} 
{\sc A.~Katok and B.~Hasselblatt}, 
{\em Introduction to the Modern Theory of Dynamical Systems}, 
Cambridge University 
Press, Cambridge, 1995.

\bibitem{Koper}
{\sc M.~T.~M.~Koper},
{\em Bifurcations of mixed-mode oscillations in a three-variable autonomous Van der Pol-Duffing model with a cross-shaped phase diagram},
Physica D, 80 (1995), pp.~72--94.

\bibitem{Kea05} 
{\sc B.~Krauskopf, H.~M.~Osinga, E.~J.~Doedel, M.~E.~Henderson, J.~Guckenheimer, A.~Vladimirsky, M.~Dellnitz, and O.~Junge}, 
{\em A survey of methods for computing (un)stable manifolds of vector fields},
Internat. J. Bifur. Chaos Appl. Sci. Engrg., 15 (2005), pp.~763--791. 

\bibitem{KrupaPopovicKopell} 
{\sc M.~Krupa, N.~Popovic, and N.~Kopell}, 
{\em Mixed-mode oscillations in three time-scale systems: A prototypical example}, 
SIAM J. Applied Dynamical Systems, 7(2008), pp.~361--420.

\bibitem{Kuehn}
{\sc C.~Kuehn},
{\em Global return maps for mixed-mode oscillations with one fast and two slow variables},
submitted 2011.

\bibitem{Dl01} 
{\sc R. de la Llave}, 
{\em A tutorial on KAM theory},
Smooth ergodic theory and its applications (Seattle, WA, 1999), Proc. Sympos. Pure Math., Amer. Math. Soc. Providence, RI. 69 (2001), pp.~175--292.

\bibitem{L88} 
{\sc R.~Lohner}, 
{\em Einschlie\ss ung der L\"osung gew\"ohnlicher Anfangs- und Randwertaufgaben und Anwendungen},
PhD thesis, Universit\"at Karlsruhe, 1988.

\bibitem{LP11} 
{\sc S.~Luzzatto and P.~Pilarczyk},
{\em Finite resolution dynamics},
Found. Comput. Math., 11 (2011), pp.~211--239. 

\bibitem{Mo66} 
{\sc R.~E.~Moore}, 
{\em Interval Analysis}, 
Prentice-Hall, Englewood Cliffs, New Jersey, 1966.

\bibitem{NJ99} 
{\sc N.~S.~Nedialkov and K.~R.~Jackson}.
{\em An interval Hermite-Obreschkoff method for computing rigorous bounds on the solution of an initial value problem for an ordinary differential equation},
Reliab. Comput., 5 (1999), pp.~289--310.

\bibitem{NJC99} 
{\sc N.~S.~Nedialkov, K.~R.~Jackson, and G.~F.~Corliss},
{\em Validated solutions of initial value problems for ordinary differential equations},
Appl. Math. Comput., 105 (1999), pp.~21--68.

\bibitem{NJP01} 
{\sc N.~S.~Nedialkov, K.~R.~Jackson, and J.~D.~Pryce},
{\em An effective High-Order Interval Method for Validating Existence and uniqueness of the Solution of an IVP for an ODE},
Reliab. Comput., 7 (2001), pp.~449--465.

\bibitem{Ne90} 
{\sc A.~Neumaier},
{\em Interval Methods for Systems of Equations},
Encyclopedia of Mathematics and its Applications 37,
Cambridge Univ. Press, Cambridge, 1990.

\bibitem{O95} 
{\sc J.~Ombach}, 
{\em Computation of the local stable and unstable manifolds},
Univ. Iagel. Acta Math., 32  (1995), pp.~129--136.

\bibitem{S90} 
{\sc C.~Sim\'{o}}, 
{\em On the Analytical and Numerical Approximation of Invariant Manifolds}, 
Les M\'{e}thodes Modernes de la Mec\'{a}nique C\'{e}leste, D Benest and C Foeschl\'{e} (eds.), Editions Fronti\`{e}rs, Paris, 1990, pp.~285-329.


\bibitem{T11} 
{\sc W.~Tucker}, 
{\em Validated numerics. A short introduction to rigorous computations}, 
Princeton University Press, Princeton, NJ, 2011, xii+138 pp.

\bibitem{Z02} 
{\sc P.~Zgliczy\'nski}, 
{\em $C^1$ Lohner algorithm},  
Found. Comput. Math. , 2  (2002), pp.~429--465. 

\bibitem{Z09} 
{\sc P.~Zgliczy\'nski}, 
{\em Covering relations, cone conditions and stable manifold theorem}, 
Journal of Differential Equations,  246 (2009), pp.~1774--1819.

\bibitem{ZG04} 
{\sc P. Zgliczy\'nski and M. Gidea}, 
{\em Covering relations for multidimensional dynamical systems}, 
J. Differential Equations, 202 (2004), pp.~32--58.

\bibitem{WZ09} 
{\sc D.~Wilczak and P.~Zgliczy\'nski}, 
{\em Computer assisted proof of the existence of homoclinic tangency for the H\'enon map and for the forced damped pendulum},  
SIAM J. Appl. Dyn. Syst.,  8  (2009), pp.~1632-1663. 

\bibitem{AUTO}
{\em AUTO: Software for continuation and bifurcation problems in ordinary differential equations},
available at {\tt cmvl.cs.concordia.ca/auto/}, 2010.

\bibitem{CGAL} {\em Computational Geometry Algorithms Library}, available at {\tt www.cgal.org/}.

\bibitem{IntLab} {\em INTLAB - the INTerval LABoratory, version 6}, available at {\tt www.ti3.tu-harburg.de/rump/intlab/}.

\bibitem{Vnode} 
{\em VNODE-LP A validated Solver for Initial Value Problems in Ordinary Differential Equations}, available at 
{\tt www.cas.mcmaster.ca/$\sim$nedialk/Software/VNODE/VNODE.shtml}.

\bibitem{Progs} Program files supplementing the paper, available at {\tt www.math.cornell.edu/$\thicksim$tjohnson}.

\end{thebibliography}
\end{document}